\documentclass[11pt]{article}
\usepackage{amsmath,amsfonts, amssymb,graphicx,geometry,authblk,xcolor,subcaption} 
\usepackage[normalem]{ulem}
\usepackage{hyperref}

\usepackage{amsthm}
\usepackage{hyperref}
\usepackage{placeins}

\newtheorem{remark}{Remark}
\newtheorem{theorem}{Theorem}[section]
\newtheorem{lemma}[theorem]{Lemma}

\newcommand{\norm}[1]{\left\lVert#1\right\rVert}

\graphicspath{{./figures/}}
\title{Optimal Design of Responsive Structures}

\title{Optimal Design of Responsive Structures}
\author[1]{Andrew Akerson}
\author[2]{Blaise Bourdin}
\author[1]{Kaushik Bhattacharya}

\affil[1]{California Institute of Technology, Pasadena CA, USA}
\affil[2]{Department of Mathematics, Louisiana State University, Baton Rouge LA, USA. Now Department of Mathematics \& Statistics, McMaster University, Hamilton ON, Canada }
\date{}

\begin{document}

\maketitle

\begin{abstract}
With recent advances in both responsive materials and fabrication techniques it is now possible to construct integrated functional structures, composed of both structural and active materials. We investigate the robust design of such structures through topology optimization. By applying a typical interpolation scheme and filtering technique, we prove existence of an optimal design to a class of objective functions which depend on the compliances of the stimulated and unstimulated states. In particular, we consider the actuation work and the blocking load as objectives, both of which may be written in terms of compliances. We study numerical results for the design of a 2D rectangular lifting actuator for both of these objectives, and discuss some intuition behind the features of the converged designs. We formulate the optimal design of these integrated responsive structures with the introduction of voids or holes in the domain, and show that our existence result holds in this setting. We again consider the design of the 2D lifting actuator now with voids. Finally, we investigate the optimal design of an integrated 3D torsional actuator for maximum blocking torque. 

\end{abstract}

\section{Introduction}
Recent advances in active or responsive materials, approaches to synthesis and fabrication, and significant applications ranging from soft robotics, wearable and prosthetic devices, microfluidics, etc.\ have led to the development of various integrated functional materials and devices (E.g.~\cite{retatl_sms_15,petal_scirep_18,yw_icra_19}). These devices integrate responsive or active materials such as shape-memory alloys, piezoelectrics, dielectric elastomers and liquid crystal elastomers with structural polymers and metals. Further, there have been rapid strides in advancing 3D printing and other synthesis technologies for responsive or active materials~\cite{aetal_acsmi_17,ketal_advmat_19,twr_ami_19,netal_sms_19,eetal_pms_16}, and in combining them with structural components to build integrated functional materials and structures~\cite{rm_aem_19}.  As the complexity and fidelity of the function, and consequently the complexity of the devices increase, it is important to develop a systematic design methodology.

Topology optimization has proven to be an extremely powerful tool in structural applications~\cite{Bendsoe-Sigmund-2003}.   The naive formulation of the classical minimum compliance problem is ill-posed (e.g.~\cite{Cherkaev-2000,Allaire-2002}).  However, it can be relaxed for example using the homogenization method~\cite{Allaire-Bonnetier-EtAl-1997,Allaire-2002} or regularized using perimeter penalization~\cite{Ambrosio-Buttazzo-1993} or phase-field approach~\cite{COCV_2003__9__19_0}.  In particular, the ``simple isotropic material with penalization (SIMP)'' interpolation in conjunction with a filter is known to be well-posed and has proven to be  extremely effective in practice~\cite{Bendsoe-Sigmund-2003}.   While topology optimization led to many real-life applications, the designs were typically complex, and manufacturing optimal designs remained a challenge.  The advent of 3D printing and similar net-shape fabrication techniques have greatly addressed these challenges and given new impetus to optimal design.  In particular,  conceptual links have been established between the multi-scale nature of topology optimization and the idea of tiling, and regularization to a method of incorporating manufacturing constraints.  These have established a pathway to 3D print (almost) optimal structures (e.g.,~\cite{petal_acmgraph_15,setal_siggraph_15,getal_Ijnme_20}).

The optimal design of structural actuators has been studied in a number of works. The design of thermomechanical actuators was originally considered by Rodrigues and Fernandes~\cite{Rodrigues1995} for 2D linear elastic solids undergoing thermal expansion. These ideas were later extended to the design of multiphysics actuators using  topology optimization methods by Sigmund~\cite{Sigmund2001,Sigmund2001a}  for the application of micro-electrical mechanical systems (MEMS). 

Since then, various researchers have considered optimal design of diverse structural actuators including soft piezoelectric microgrippers~\cite{Ruiz2018}, magnetic actuators~\cite{Park2012}, and electro-fluid-thermal compliant actuators~\cite{Yoon2012}. In these studies, the actuator is characterized by three primary objectives. The first is the actuation work which is equivalent to the flexibility or displacement on actuation. The second is the blocking load, or the applied load that can nullify the actuation. The third is the ``workpiece'' objective, which balances  flexibility and stiffness of the structure using a spring attached to a point of interest.

In this work, we provide a mathematical framework to explore the use of topology optimization for the design of integrated responsive structures. Specifically, we consider an actuating structure composed of both an active material which can deform or change modulus in response to a stimulus, and a passive material. We formulate the design of these structures as an optimization problem for a general class of objective functions which are dependent on the compliances in the passive (unstimulated) and active (stimulated) states. Modified with a generalized SIMP interpolation and density filter, we prove existence of an optimal design. It can be shown that all three of the objectives described in the previous discussion can be written as functions of compliance, and thus satisfy the requirements for existence. In particular, we study the actuation work and blocking load objectives. The first is the difference in compliances while we show that the second is equivalent to the ratio of compliances. We provide numerical examples for both of these objectives, and discuss intuition behind the converged designs.

We begin in Section~\ref{sec:back} by reviewing the ill-posed minimum compliance problem, and recall how a SIMP interpolation and filtering technique may lead to a well-posed problem. In Section~\ref{sec:gen_obejctive} we introduce the energy functional for the responsive structure, and formulate an optimal design problem in which the objective is dependent on the compliances of both the stimulated and un-stimulated structure. By regularizing with a SIMP interpolation and density filter, we then prove existence of solutions to this optimal design problem. We continue in Section~\ref{sec:objectives} where we discuss objective functions used to characterize actuating systems, namely a generalized ``workpiece'' objective, the actuation work, and the blocking load. We show that all of these can be written in terms of compliances and thus satisfy the requirements for our existence result. Additionally, we show that the latter two appear as limiting cases of the generalized ``workpiece'' objective. In Section~\ref{sec:examples} we consider numerical examples of actuating structures. We begin in Section~\ref{sec:work} with the actuation work objective and consider the 2D design of bimorph lifting actuators for varying elastic moduli ratios of passive and responsive material, volume fractions of active material, and domain aspect ratios. Next, in Section~\ref{sec:blocking}, we consider the design of identical actuator setups now optimized for the blocking load. In Section~\ref{sec:voids} we consider the introduction of holes or voids in the domain, and show that the existence result continues to hold for the blocking load objective. We consider the 2D design of a lifting actuator with voids in Section~\ref{sec:voids2D}. Then, in Section~\ref{sec:voids3D}, we demonstrate the formulation in a 3D setting by considering the optimal design of a torsional actuator. Finally, in Section~\ref{sec:con}, we discuss challenges and directions for further studies.

\section{Background: Compliance optimization} \label{sec:back}


We briefly recall the classical minimum compliance problem~\cite{Bendsoe-Sigmund-2003}. Consider a bounded domain $\Omega\subset \mathbb{R}^n$, $n=2,3$, subject to a known traction $f$ on a part $\partial_f\Omega \subset \partial \Omega$ of its boundary, and prescribed displacement $u_0$ on $\partial_u\Omega \subset \partial \Omega$. The domain is partitioned into two regions, $D_1$ and $D_2$, occupied respectively by two known linear elastic materials of moduli $\mathbb{C}_1$ and $\mathbb{C}_2$. We seek the arrangement of the regions $D_1$ and $D_2$ that minimize the compliance,
\begin{equation}
	\label{eq:minCompliance}
	\inf_{(D_1,D_2)\in \mathcal{C}_A} \mathcal{C}(D_1,D_2):=\int_{\partial_f \Omega} f\cdot u\, dS,
\end{equation}
where $\mathcal{C}_A$ is a set of admissible designs subject to inclusions and volume fraction constraints, and $u$ the equilibrium displacement solution of a linearized elasticity problem. 
%
%
It is well known that problem~\eqref{eq:minCompliance} is ill-posed: its solution consists of fine mixtures of regions $D_1$ and $D_2$ instead of a ``conventional'' design.
In the {homogenization approach}~\cite{Allaire-Bonnetier-EtAl-1997,Allaire-2002}, the design variables and the state equation are reformulated in terms of the so-called $G$-closure or the set of all Hooke's laws achievable by mixtures of materials $D_1$ and $D_2$.
In this sense, the design of optimal structures reduces to that of optimal metamaterials.  Metamaterials with optimal properties have been constructed explicitly~\cite{Francfort-Murat-1986,Milton-1986a,Cherkaev-2000,Milton-2002a,Allaire-2002} for well-ordered non-degenerate materials. However, optimal metamaterials are not known explicitly in general.
Further, homogenization-based approaches do not always lead to manufacturable designs even with additive manufacturing, since optimality generally {\it requires multiple length-scales}~\cite{Francfort-Murat-1986,Milton-1986a,Allaire-Aubry-1999a,Bourdin-Kohn-2007}, which can make manufacturing optimal structures challenging.


The issue of manufacturability can be tackled by enforcing geometric constraints~\cite{Allaire-Bogosel-2018a,Qian_2017,Wang_2017,Allaire_2017,Allaire-Dapogny-EtAl-2017a,Allaire-Jouve-EtAl-2016a}, or by seeking near-optimal designs with reduced complexity.
The SIMP approach~\cite{Bendsoe-Sigmund-1999} relaxes the problem to ``grey-scale'' designs through a material density $\phi$ taking values in $[0,1]$. The effective Hooke's law is then interpolated to $\mathbb{C} = \phi^p\mathbb{C}_1+ (1-\phi^p)\mathbb{C}_2$, $p>1$.  
In some cases, this is equivalent to the homogenization approach within a family of sub-optimal micro-geometries~\cite{Bendsoe-Sigmund-1999}.  In any case, this approach does not lead to a well-posed problem.  Practically, a SIMP-based implementation alone suffers from mesh dependencies (the smallest feature detected depends on the mesh size) and checkerboards (design patterns at the scale of the finite element mesh that are poorly approximated by low order finite elements). Thus, they are commonly used in conjunction with a filtering technique~\cite{Sigmund-1997}, where either the sensitivities or densities are averaged during optimization, which has been shown to lead to a well-posed problem~\cite{Bourdin-2001}. 
Roughly speaking, introducing a non-local term in the response function (the filter) provides compactness of minimizing sequences of designs which combined with the lower semi-continuity of the objective function, is sufficient to prove existence of ``classical'' solutions. We will borrow these ideas for the optimal design of responsive structures to formulate a well-posed problem.

\section{Optimizing responsive structures} \label{sec:gen_obejctive}

\subsection{Responsive material }

A responsive material is one that changes its shape and/or stiffness in response to a stimulus.  These could include activated, ferroelectric, and magnetostrictive materials.  The corresponding elastic energy density may be described as
\begin{equation}
W(\varepsilon, S) := \frac{1}{2} (\varepsilon(u) - \varepsilon^*(S) ) \cdot {\mathbb C} (S) (\varepsilon(u) - \varepsilon^*(S)),
\end{equation}
where $\varepsilon(u) = (\nabla u + \nabla u^T)/2$ is the linearized strain, $S \in [0, 1]$ the stimulus (assumed here to be a scalar), $\varepsilon^*(S)$ is the stimulus-dependent actuation or spontaneous strain, and $\mathbb C(S)$ is the  possibly stimulus-dependent elastic modulus.  We assume that $\varepsilon^*(0)=0$.

\subsection{Optimal design}

Consider an integrated functional structure occupying a bounded region $\Omega \subset \mathbb{R}^n$ of volume $V$, consisting of a structural material and a responsive material.  Let $\chi_s, \chi_r : \Omega \mapsto \{0, 1\}$ be the characteristic functions of the regions the structural and responsive materials occupy. $\Phi := (\chi_s, \chi_r)$ then describes the design. The structural material may either be a stiff framework or a soft binder.  In this section, we assume that there are no voids, i.e., $\chi_s + \chi_r = 1$.

Subject to a traction  $f \in L^2(\partial_f \Omega)$ on  $\partial_f\Omega$ and displacement $u_0$ on $\partial_u\Omega$, the energy function describing this structure for a given displacement field $u$ under stimulus $S \in [0, 1]$  is
\begin{equation}
\label{eq:ElastEnergy}
\mathcal{E} (\Phi, u, S):= \int_\Omega\frac{1}{2} \left[ \chi_s\varepsilon(u) \cdot \mathbb C_s \varepsilon(u) 
+ \chi_r (\varepsilon(u) - \varepsilon^*) \cdot \mathbb C_r(S)  (\varepsilon(u) - \varepsilon^*) \right]\, dx - \int_{\partial_f \Omega}f\cdot u\, ds,
\end{equation}
where $\mathbb{C}_s$ and $\mathbb{C}_r$ denote respectively the Hooke's laws of the structural and responsive materials. We assume that these Hooke's laws are non-degenerate in the sense that there exist 4 positive constants $0< m_s,m_r,M_s,M_r$ such that 
\begin{equation}
	\label{eq:NonDegenerateHookesLaws}
	m_{r,s} \le \mathbb{C}_{r,s}\eta\cdot\eta\, \le M_{r,s}
\end{equation}
for any symmetric $n$-dimensional second order tensor $\eta$ with unit norm. The compliance of a design is
\begin{equation}
	\label{eq:defCompliance}
	\mathcal{C}(\Phi,S):= \int_{\partial_f \Omega}f\cdot u \, ds,
\end{equation}
where $u$ is the displacement given by the equilibrium condition
\begin{equation} \label{eq:u_min}
u = \text{arg}\min_{u \in \mathcal{U}} \  \mathcal{E}(\Phi, u, S),
\end{equation} 
where
\begin{equation}
\mathcal{U} := \{u \in W^{1, 2}(\Omega) : u = u_0 \text{ on } \partial_u \Omega \}.
\end{equation}
Equivalently, $u$ satisfies
\begin{equation}
\label{eq:responsiveEquilibrium}
G(\Phi,S;u):= -\mathrm{div} \left[\chi_s(x)\mathbb{C}_s\varepsilon(u) + \chi_r(x)\mathbb{C}_r(\varepsilon(u) -  \varepsilon^*(S) )\right] = 0,
\end{equation}
subject to the boundary conditions. 
The task is to find the design $\Phi$ that minimizes an objective function, which we assume to be dependent on the compliances of two states with different stimuli. Thus, we consider a class of optimization problems
\begin{equation} \label{eq:optimization_problem}
\inf_{\Phi \in \mathcal{D}} \ \mathcal{O}(\Phi) := \bar{\mathcal{O}}(\mathcal{C}(\Phi, S_1), \mathcal{C}(\Phi, S_2) ),
\end{equation}
where $\bar{\mathcal{O}} : \mathbb{R} \times \mathbb{R} \mapsto \mathbb{R}$ is a given continuous function, amongst the set of allowable designs:
\begin{equation}
{\mathcal D} = \{\Phi: \chi_r + \chi_s = 1 \text{ on } \Omega, \ \int_\Omega \chi_r \ dV \leq \bar{V}_r \}.
\end{equation}
 Here, we have specified the allowable designs such that the materials occupy the whole domain and consider a restriction on the volume of responsive material, where  $\bar{V}_r \leq V$ is the maximum allowed. The above problem is often ill-posed, suffering from the same issues as that of the standard compliance optimization problem in~\eqref{eq:minCompliance}. Thus, we introduce a SIMP interpolation and a filter as discussed in the previous section.

\subsection{Reformulation using interpolation and regularization}

Consider the relaxed energy functional for the responsive structure with a SIMP interpolation of penalty factor $p > 1$, 
\begin{multline}
\mathcal{E}_f(\phi, u, S) := \int_{\Omega} \frac{1}{2}\left[  (1 - (F * \phi)^p) \mathbb{C}_s \varepsilon(u) \cdot \varepsilon(u) + (F * \phi)^p \mathbb{C}_r(S)(\varepsilon(u) - \varepsilon^*(S)) \cdot (\varepsilon(u) - \varepsilon^*(S)) \right] \ dx \\
- \int_{\partial_f \Omega} f \cdot u \ ds,
\end{multline}
where $F$ is the filter function of characteristic length $R_f > 0$ satisfying
\begin{equation}
\begin{aligned}
& F \in  W^{1, 2}(\mathbb{R}^n), \\ 
& \text{Supp}(F) \subset B_{R_f}, \\
& F \geq 0 \ \text{ a.e. in } \, B_{R_f}, \\
& \int_{B_{R_f}} F \ dx = 1,
\end{aligned}
\end{equation}
where $B_{R_f}$ denotes the open ball of radius $R_f$ of $\mathbb{R}^n$,
and we define the convolution over the bounded region $\Omega$ as
\begin{equation} \label{eq:conv}
(F*\phi) (x) := \int_{\Omega} F(x - y) \phi(y) \ dy.
\end{equation}
The field $\phi: \Omega \mapsto [0, 1]$ describes the topology, with regions of $\phi = 0$ and $\phi = 1$ corresponding to passive and active material, respectively. We assume the transformation strain $\varepsilon^*(S) \in L^2(\Omega)$. Notice that since the integral in \eqref{eq:conv} is over $\Omega$, the filtered density near the boundary will not be able to take values near $1$. In practice, we renormalize the convolution following the lines of~\cite{Bourdin-2001} to avoid such boundary effects, as discussed in Section~\ref{sec:examples}. The compliance of a design is, again, 
\begin{equation}
\label{eq:defCompliance_phi}
\mathcal{C}(\phi,S):= \int_{\partial_f \Omega}f\cdot u \, ds,
\end{equation}
where $u$ is the displacement associated with the design $\Phi$ and stimulus $S$ minimizing $\mathcal{E}_f(\phi, u, S)$.
It should be noted that under the assumptions on $\mathbb{C}_s$ and $\mathbb{C}_r$ in~\eqref{eq:NonDegenerateHookesLaws}, $u$ is the unique solution of the Euler-Lagrange equations.
\begin{equation} \label{eq:u_min_f} 
\mathcal{Q}_f(\phi, u, v,  S) = 0 \quad  \forall \ v \in \mathcal{U}_0,
\end{equation}
with
\begin{multline} \label{eq:weak_eq}
\mathcal{Q}_f(\phi, u, v,  S) := \int_{\Omega} \left[  (1 - (F * \phi)^p) \mathbb{C}_s \varepsilon(u) \cdot \varepsilon(v) + (F * \phi)^p \mathbb{C}_r(S) (\varepsilon(u) - \varepsilon^*(S)) \cdot (\varepsilon(v)) \right] \ dx \\
- \int_{\partial_f \Omega} f \cdot v \ ds,
\end{multline}  
and
\begin{equation}
\mathcal{U}_0 := \{u \in W^{1, 2}(\Omega) : u = 0 \text{ on } \partial_u \Omega \}.
\end{equation}
 We again consider the class of optimization problems
\begin{equation} \label{eq:optimization_problem_filter}
\inf_{\phi \in \mathcal{D}_f} \ \mathcal{O}(\phi) := \bar{\mathcal{O}}(\mathcal{C}(\phi, S_1), \mathcal{C}(\phi, S_2) ),
\end{equation}
where we optimize over the set
\begin{equation}
\mathcal{D}_f := \left \{ \phi: \phi(x) \in [0, 1] \ \text{ a.e. on } \Omega, \int_{\Omega} \phi \ dx \leq \bar{V}_r \right \}.
\end{equation}
\begin{remark} \label{ellipticity_remark}
	Ellipticity: From the definition of $\mathcal{D}$, the properties of the Hooke's laws~\eqref{eq:NonDegenerateHookesLaws}, and the properties of the filter, there exists a constant $m > 0$, only depending on $\Omega$ and $S$ such that for any $\phi$ and $u\in\mathcal{U}$ the following holds:
	\begin{equation}
	\int_{\Omega} \frac{1}{2} \left[  (1 - (F * \phi)^p) \mathbb{C}_s \varepsilon(u) \cdot \varepsilon(u) + (F * \phi)^p \mathbb{C}_r(S) \varepsilon(u) \cdot \varepsilon(u) \right] \ dx \geq m \norm{u}^2_{W^{1,2}(\Omega)}.
	\end{equation}
\end{remark}

\subsection{Existence of solutions}

We establish the existence of solutions to~\eqref{eq:optimization_problem} through the following theorem in this section 

\begin{theorem} \label{thm:ext}
	Recall the definition of the compliances from~\eqref{eq:defCompliance_phi}, and set
	\begin{equation}
	\mathcal{O}(\phi) := \bar{\mathcal{O}}(\mathcal{C}(\phi, S_1), \mathcal{C}(\phi, S_2) ),
	\end{equation}
	where $\bar{\mathcal{O}}$ is bounded below and continuous. There exists a $\bar{\phi} \in \mathcal{D}_f$ such that,
	\begin{equation} \label{eq:O_inf}
	\mathcal{O}(\bar{\phi}) = \inf_{\phi \in \mathcal{D}_f } \ \mathcal{O}(\phi).
	\end{equation}
\end{theorem}

We need two lemmas to prove this theorem.  The first establishes the weak continuity of the solutions to the elliptic problem and the second the weak continuity of the compliance.

\begin{lemma} \label{lem:00}
	Let $\{u_k\} \subset \mathcal{U}$ be the sequence of equilibrium solutions to~\eqref{eq:u_min_f} corresponding to sequence $ \{ \phi_k \} \subset \mathcal{D}_f$ for some fixed $S$. If
	\begin{equation}
	(F * \phi_k)^p \rightarrow (F * \bar{\phi})^p \ \text{uniformly on } \Omega \text{ when } k \rightarrow + \infty
	\end{equation}
	then
	\begin{equation}
	\begin{aligned}
	u_{k} \rightharpoonup \bar{u} \text{ in } W^{1,2}(\Omega) \text{ when } k \rightarrow + \infty, \\
	\end{aligned}
	\end{equation}
	up to a subsequence, where $\bar{u} \in \mathcal{U}$ is the equilibrium configuration corresponding to $\bar{\phi} \in \mathcal{D}_f$.
\end{lemma}
\begin{proof}
	We will first show by compactness that there exists a $u^{\infty} \in \mathcal{U}$ such that $u_k \rightharpoonup u^{\infty}$ in $W^{1,2}(\Omega)$. Then, we will show that we must have $u^{\infty} = \bar{u}$.
	
	Since $u_k$ is the equillibrium solution corresponding to $\phi_k$ for some fixed $S$, it satisfies
	\begin{equation}
	u_k = \text{arg}\min_{u \in \mathcal{U}} \ \mathcal{E}_f(\phi_k, u, S),
	\end{equation}
	and for any $\tilde{u} \in \mathcal{U}$, we have
	\begin{equation}
	\mathcal{E}_f (\phi_k, u_k, S) \leq \mathcal{E}_f (\phi_k, \tilde{u}, S).
	\end{equation}
	Furthermore, 
	\begin{multline}
	\mathcal{E}_f (\phi_k, \tilde{u}, S) \leq \int_{\Omega} \frac{1}{2} \left[  \mathbb{C}_s \varepsilon(\tilde{u}) \cdot \varepsilon(\tilde{u}) + \mathbb{C}_r(S) (\varepsilon(\tilde{u}) - \varepsilon^*(S)) \cdot (\varepsilon(\tilde{u}) - \varepsilon^*(S)) \right] \ dx \\
	- \int_{\partial_f \Omega} f \cdot \tilde{u} \ ds = M
	\end{multline}
	where $M$ is some constant, independent of $k$. So,
	\begin{equation}
	\mathcal{E}_f (\phi_k, u_k, S) \leq M.
	\end{equation}
	Now, expanding the energy functional
	\begin{equation}
	\begin{aligned}
	\mathcal{E}_f (\phi_k, u_k, S) = &\int_{\Omega} \frac{1}{2} \big[(1 - (F * \phi_k)^p) \mathbb{C}_s \varepsilon(u_k) \cdot \varepsilon(u_k) + (F * \phi_k)^p \mathbb{C}_r(S) \varepsilon(u_k) \cdot \varepsilon(u_k)   \\
	& \qquad + (F * \phi_k)^p \mathbb{C}_r(S)\varepsilon^*(S) \cdot \varepsilon^*(S)  -  2(F * \phi_k)^p \mathbb{C}_r (S) \varepsilon^*(S)  \cdot \varepsilon(u_k) \big] \ dx \\
	- &\int_{\partial_f \Omega} f \cdot u_k \ ds,
	\end{aligned}
	\end{equation}
	and using the ellipticity from Remark \ref{ellipticity_remark}
	\begin{equation}
	\begin{aligned}
	m \norm{u_k}^2_{W^{1,2}(\Omega)} -  \left| \int_{\Omega} (F * \phi_k)^p \mathbb{C}_r (S) \varepsilon^*(S) \cdot \varepsilon(u_k) \ dx \right | - \int_{\partial_f \Omega} f \cdot u_k \ ds &\leq \mathcal{E}_f (\phi_k, u_k, S), \\
	m \norm{u_k}^2_{W^{1,2}(\Omega)} -  \int_{\Omega} (F * \phi_k)^p  \left| \mathbb{C}_r (S) \varepsilon^*(S) \cdot \varepsilon(u_k) \right |  \ dx - \int_{\partial_f \Omega} f \cdot u_k \ ds &\leq \mathcal{E}_f (\phi_k, u_k, S), \\ 
	m \norm{u_k}^2_{W^{1,2}(\Omega)} -  \int_{\Omega} (F * 1)^p  \left| \mathbb{C}_r (S) \varepsilon^*(S) \cdot \varepsilon(u_k) \right |  \ dx - \int_{\partial_f \Omega} f \cdot u_k \ ds &\leq \mathcal{E}_f (\phi_k, u_k, S), \\
	m \norm{u_k}^2_{W^{1,2}(\Omega)} -  \int_{\Omega} \left|\mathbb{C}_r (S) \varepsilon^*(S) \cdot \varepsilon(u_k) \right |  \ dx - \int_{\partial_f \Omega} f \cdot u_k \ ds &\leq \mathcal{E}(\phi_k, u_k, S), \\
	m \norm{u_k}^2_{W^{1,2}(\Omega)} -  \norm{\mathbb{C}_r (S) \varepsilon^*(S)}_{L^2(\Omega)}  \norm{\varepsilon(u_k)}_{L^2(\Omega)} -\int_{\partial_f \Omega} f \cdot u_k \ ds &\leq \mathcal{E}_f (\phi_k, u_k, S), \\ 
	m \norm{u_k}^2_{W^{1,2}(\Omega)} - c \norm{u_k}_{W^{1,2}( \Omega)} - \int_{\partial_f \Omega} f \cdot u_k \ ds &\leq \mathcal{E}_f (\phi_k, u_k, S) 
	\end{aligned}
	\end{equation}
	for some constants $m,c > 0$, independent of $k$. Additionally,
	\begin{equation}
	\begin{aligned}
	m \norm{u_k}^2_{W^{1,2}(\Omega)} - c \norm{u_k}_{W^{1,2}( \Omega)}	- \left | \int_{\partial_f \Omega} f \cdot u_k \ ds \right |  &\leq \mathcal{E}_f (\phi_k, u_k, S), \\
	m \norm{u_k}^2_{W^{1,2}(\Omega)} - c \norm{u_k}_{W^{1,2}( \Omega)}	- \norm{f}_{L^2(\partial_f \Omega)} \norm{u_k}_{L^2(\partial_f \Omega)} &\leq \mathcal{E}(\phi_k, u_k, S), \\
	m \norm{u_k}^2_{W^{1,2}(\Omega)} - c \norm{u_k}_{W^{1,2}( \Omega)} - \norm{f}_{L^2(\partial_f \Omega)} \norm{u_k}_{L^2(\partial \Omega)} &\leq \mathcal{E}_f (\phi_k, u_k, S), \\
	m \norm{u_k}^2_{W^{1,2}(\Omega)} - c \norm{u_k}_{W^{1,2}( \Omega)} - a \norm{u_k}_{W^{1,2}(\Omega)} &\leq \mathcal{E}_f (\phi_k, u_k, S),
	\end{aligned}
	\end{equation}
	for some constant $a > 0$. Then,
	\begin{equation}
	m \norm{u_k}^2_{W^{1,2}(\Omega)} - b \norm{u_k}_{W^{1,2}(\Omega)} \leq M \implies \norm{u_k}_{W^{1,2}(\Omega)} \leq d,
	\end{equation}
	for some constant $b > 0$, where $d > 0$ is a constant independent of $k$. Thus, $u_k$ is a bounded sequence in $W^{1,2}(\Omega)$, and there exists a $u^\infty \in \mathcal{U}$ such that
	\begin{equation}
	u_{k} \rightharpoonup u^{\infty} \text{ in } W^{1,2}(\Omega) \text{ when } k \rightarrow + \infty, \\
	\end{equation}
	up to a subsequence. Next, consider $\bar{u} \in \mathcal{U}$ such that
	\begin{equation}
	\bar{u} = \text{arg}\min_{u \in \mathcal{U}}  \mathcal{E}_f (\bar{\phi}, u, S).
	\end{equation}
	Then
	\begin{equation} \label{eq:u_bar_less_u_infty}
	\mathcal{E}_f (\bar{\phi}, \bar{u}, S) \leq \mathcal{E}_f (\bar{\phi}, u^{\infty}, S).
	\end{equation}
	Similarly,
	\begin{multline}
	\mathcal{E}_f(\phi_k, u_k, S) \leq \mathcal{E}_f(\phi_k, \bar{u}, S) = \mathcal{E}_f(\phi_k, \bar{u}, S) - \mathcal{E}_f(\bar{\phi}, \bar{u}, S) + \mathcal{E}_f(\bar{\phi}, \bar{u}, S) \\
	- \mathcal{E}_f(\bar{\phi}, u_k, S) + \mathcal{E}_f(\bar{\phi}, u_k, S),
	\end{multline}
	or
	\begin{equation}
	\mathcal{E}_f(\bar{\phi}, u_k, S) \leq \mathcal{E}_f(\bar{\phi}, \bar{u}, S) + \mathcal{E}_f(\phi_k, \bar{u}, S) - \mathcal{E}_f(\bar{\phi}, \bar{u}, S)  + \mathcal{E}_f(\bar{\phi}, u_k, S) - \mathcal{E}_f(\phi_k, u_k, S).
	\end{equation}
	Then taking limits, and using the strong convergence of the convolution gives
	\begin{equation}
	\lim_{k \rightarrow \infty} \ \mathcal{E}_f(\bar{\phi}, u_k, S)  \leq \mathcal{E}_f(\bar{\phi}, \bar{u}, S).
	\end{equation}
	The convexity of the energy integrand in $\nabla u$ and $u$ for a given $\phi$ and $S$ gives lower semi-continuity of our energy function
	\begin{equation}
	\mathcal{E}_f(\bar{\phi}, u^{\infty}, S) \leq \lim_{k \rightarrow \infty} \ \mathcal{E}_f(\bar{\phi}, u_k, S),
	\end{equation}
	so
	\begin{equation}
	\mathcal{E}_f(\bar{\phi}, u^{\infty}, S) \leq \mathcal{E}_f(\bar{\phi}, \bar{u}, S).
	\end{equation}
	Then from~\eqref{eq:u_bar_less_u_infty},
	\begin{equation}
	\mathcal{E}_f(\bar{\phi}, u^{\infty}, S) = \mathcal{E}_f(\bar{\phi}, \bar{u}, S).
	\end{equation}
	From the uniqueness of the minimizer of $\mathcal{E}_f(\bar{\phi}, \cdot, S)$ we have
	\begin{equation}
	u^\infty = \bar{u}.
	\end{equation}
	Then, as desired,
	\begin{equation}
	u_{k} \rightharpoonup \bar{u} \text{ in } W^{1,2}(\Omega) \text{ when } k \rightarrow + \infty. \\
	\end{equation}
\end{proof}

\begin{lemma} \label{lem:01}
	Let $\{u_k\} \subset \mathcal{U}$ be the sequence of equilibrium solutions corresponding to sequence $\{ \phi_k \} \subset \mathcal{D}_f$ for some fixed $S$. If
	\begin{equation}
	\begin{aligned}
	u_{k} \rightharpoonup \bar{u} \text{ in } W^{1,2}(\Omega) \text{ when } k \rightarrow + \infty, \\
	\end{aligned}
	\end{equation}
	where $\bar{u} \in \mathcal{U}$ is the equilibrium configuration corresponding to $\bar{\phi} \in \mathcal{D}_f$, then
	\begin{equation}
	\lim_{k \rightarrow \infty} \ \mathcal{C}(\phi_k, S) = \mathcal{C}(\bar{\phi}, S).
	\end{equation}
\end{lemma}

\begin{proof}
	Because $\bar{u}$ satisfies equilibrium~\eqref{eq:u_min}, and $(\bar{u} - u_k) \in \mathcal{U}_0$,
	\begin{equation}
	\mathcal{Q}_f(\bar{\phi}, \bar{u}, (\bar{u} - u_k) , S) = 0.
	\end{equation}
	Expanding and using the definition of the compliance~\eqref{eq:defCompliance_phi}, this can be written as
	\begin{equation}
	\begin{aligned} [b]
	\mathcal{C}(\bar{\phi}, S) = \  &\mathcal{C}(\phi_k, S) +  \int_{\Omega} \left[  (1 - (F * \bar{\phi})^p) \mathbb{C}_s \varepsilon(\bar{u})  + (F * \bar{\phi})^p \mathbb{C}_r(S) (\varepsilon(\bar{u}) - \varepsilon^*(S)) \right] \cdot (\varepsilon(\bar{u}) - \varepsilon({u_k})) \ dx. \\
	\end{aligned}
	\end{equation}
	Taking limits and noting that $u_k \rightharpoonup \bar{u}$ in $W^{1,2}(\Omega)$ gives
	\begin{equation}
	\mathcal{C}(\bar{\phi}, S) = \lim_{k \rightarrow \infty} \ \mathcal{C}(\phi_k, S).
	\end{equation}
\end{proof}

We are now ready to prove Theorem \ref{thm:ext} or the existence of minimizers to the optimization problem~\eqref{eq:optimization_problem_filter}.
\begin{proof}
	Let $ \{ \phi_k \} \subset \mathcal{D}_f$ be a minimizing sequence for~\eqref{eq:O_inf}. $\mathcal{D}_f$ implies that $\phi_k$ is uniformly bounded in $L^2(\Omega)$ and thus there exists $\bar{\phi} \in \mathcal{D}_f$ such that
	\begin{equation}
	\phi_k \rightharpoonup \bar{\phi} \ \text{ in } L^2(\Omega) \text{ when } k \rightarrow + \infty,
	\end{equation}
	up to a subsequence. Because $F \in L^2(\mathbb{R}^n)$,  
	\begin{equation}
	(F * \bar{\phi})(x) - \lim_{k \rightarrow \infty} \ (F * \phi_k)(x)  = \lim_{k \rightarrow \infty} \int_{\Omega} F(x - y) \left( \bar{\phi}(y) - \phi_k(y) \right) \ dy = 0.
	\end{equation}
	Since this holds for all $x \in \Omega$, 
	\begin{equation}
	F * \phi_k \rightarrow F * \bar{\phi} \ \text{uniformly on } \Omega \text{ when } k \rightarrow + \infty.
	\end{equation}
	Because $(F * \phi)(x)$ is bounded for all $\phi \in \mathcal{D}_f$, 
	\begin{equation}
	(F * \phi_k)^p \rightarrow (F * \bar{\phi})^p \ \text{uniformly on } \Omega \text{ when } k \rightarrow + \infty.
	\end{equation}
	Let $u_{1k}, u_{2k} \in \mathcal{U}$ be the equilibrium solutions to~\eqref{eq:u_min_f} corresponding to $\phi_k$ for $S = S_1$ and $S = S_2$, respectively:
	\begin{equation}
	u_{1k} = \arg \min_{u \in \mathcal{U}} \ \mathcal{E}_f (\phi_k, u, S_1), \qquad u_{2k} = \arg \min_{u \in \mathcal{U}} \ \mathcal{E}_f (\phi_k, u, S_2).
	\end{equation}
	 Then, from Lemma \ref{lem:00}, 
	\begin{equation}
	\begin{aligned}
	u_{1k} \rightharpoonup \bar{u}_{1} \text{ in } W^{1,2}(\Omega) \text{ when } k \rightarrow + \infty, \\
	u_{2k} \rightharpoonup \bar{u}_{2} \text{ in } W^{1,2}(\Omega) \text{ when } k \rightarrow + \infty,
	\end{aligned}
	\end{equation}
	where $\bar{u}_{1}, \bar{u}_{2} \in \mathcal{U}$ are the equilibrium configurations corresponding to $\bar{\phi}$ for $S = S_1$ and $S = S_2$:
	\begin{equation}
	\bar{u}_1 = \arg \min_{u \in \mathcal{U}} \ \mathcal{E}_f (\bar{\phi}, u, S_1), \qquad \bar{u}_2 = \arg \min_{u \in \mathcal{U}} \ \mathcal{E}_f (\bar{\phi}, u, S_2).
	\end{equation}
	From  Lemma \ref{lem:01},
	\begin{equation}
	\lim_{k \rightarrow \infty} \ \mathcal{C}(\phi_k, S_1) = \mathcal{C}(\bar{\phi}, S_1),  \qquad 	\lim_{k \rightarrow \infty} \ \mathcal{C}(\phi_k, S_2) = \mathcal{C}(\bar{\phi}, S_2).
	\end{equation}
	It follows
	\begin{equation}
	\lim_{k \rightarrow \infty} \ \bar{\mathcal{O}}(\mathcal{C}(\phi_k, S_1), \mathcal{C}(\phi_k, S_2)) = \bar{\mathcal{O}}(\mathcal{C}(\bar{\phi}, S_1), \mathcal{C}(\bar{\phi}, S_2)).
	\end{equation}
	Therefore,
	\begin{equation}
	\lim_{k \rightarrow \infty} \ \mathcal{O}(\phi_k) = \mathcal{O}(\bar{\phi}).
	\end{equation}
	Since $\phi_k$ is a minimizing sequence, 
	\begin{equation}
	\mathcal{O}(\bar{\phi}) = \inf_{\phi \in \mathcal{D}} \ \mathcal{O}(\phi).
	\end{equation}
\end{proof}

\subsection{Sensitivities through the adjoint method}
We  solve the optimal design problem using a gradient-based approach.  To do so, we need to compute the directional derivative of the objective function with respect to a design changes. To this end, we utilize an adjoint approach. Consider $u_1$ and $u_2$ associated with $S = S_1$ and $S = S_2$ which satisfy~\eqref{eq:u_min_f} for some design $\phi \in \mathcal{D}_f$. To find the directional derivative of some functional $\mathcal{F}(\phi, u_1, u_2)$, we introduce the augmented objective $\mathcal{L}(\phi,u_1,u_2,\lambda_1,\lambda_2) = \mathcal{F}(\phi, u_1, u_2)$ for any $\lambda_1,\lambda_2 \in \mathcal{U}_0$,
\begin{equation}
\label{eq:lagangianResponsiveSensitivity}
\mathcal{L}(\phi,u_1,u_2, \lambda_1, \lambda_2) := \mathcal{F} (\phi, u_1, u_2) + \mathcal{Q}_f(\phi, u_1, \lambda_1, S_1) + \mathcal{Q}_f(\phi, u_2, \lambda_2, S_2).
\end{equation}
One can easily show that the directional derivative of $\mathcal F$ in the direction $\tilde{\phi}$ is
\begin{equation}
\label{eq:derivativeResponsiveSensitivity}
\mathcal{F}^\prime (\phi) \tilde{\phi} = \mathcal{F}_{,\phi} (\phi, u_1, u_2) \tilde{\phi} + \mathcal{Q}_{f,\phi} (\phi, u_1, \lambda_1^*, S_1) \tilde{\phi} + \mathcal{Q}_{f,\phi} (\phi, u_2, \lambda_2^*, S_2) \tilde{\phi},
\end{equation}
where $\lambda^*_1, \lambda^*_2 \in \mathcal{U}_0$ are solutions of the uncoupled adjoint equations
\begin{equation}
\label{eq:adjointresponsiveSensitivity}\
\begin{cases}
\displaystyle
\mathcal{F}_{, u_1} (\Phi, u_1, u_2) \tilde{u} + \mathcal{Q}_{f, u_1} (\phi, u_1, \lambda_1^*, S_1) \tilde{u}  = 0 & \forall \  \tilde{u} \in \mathcal{U}_0,\\
\displaystyle
\mathcal{F}_{, u_2} (\Phi, u_1, u_2) \tilde{u} + \mathcal{Q}_{f, u_2} (\phi, u_2, \lambda_2^*, S_2) \tilde{u}  = 0 & \forall \  \tilde{u} \in \mathcal{U}_0.\\
\end{cases}
\end{equation}

\section{Objective functions} \label{sec:objectives}

\subsection{General workpiece objective}
We discuss a variety of objective functions used to characterize actuating systems. 
We define the general \emph{ workpiece objective} to be 
\begin{equation} \label{eq:workpiece_gen}
	\bar{\mathcal{O}} (\mathcal{C}(\phi, 0), \mathcal{C}(\phi, 1) ) = \frac{\kappa \mathcal{C}(\phi, 1) + 1}{\kappa \mathcal{C}(\phi, 0) + 1},
\end{equation}
where $\kappa \in (0, +\infty)$ is a parameter.  In the case where $f$ is a point load in direction $\hat{n}$ at point $x_0$, this objective is equivalent to maximizing the force carried by a linear elastic spring in direction $\hat{n}$ of stiffness $\kappa$ attached at $x_0$ (see Appendix~\ref{ap:workpiece}).   Further, this objective is dependent on the compliances of the stimulated and unstimulated states and therefore satisfies the conditions of  Theorem \ref{thm:ext}. 
This workpiece objective has interesting limits when the parameter tends to either zero or infinity.  

First, consider the limit of small $\kappa$. Using the Taylor expansion of~\eqref{eq:workpiece_gen} about $\kappa = 0$,
\begin{equation}
	\frac{\kappa \mathcal{C}(\phi, 1) + 1}{\kappa \mathcal{C}(\phi, 0) + 1} \approx 1 + \kappa \left( \mathcal{C}(\phi, 1) - \mathcal{C}(\phi, 0)\right).
\end{equation} 
Thus, for small $\kappa$ the workpiece objective is equivalent to the difference in compliance.  We show in Section~\ref{sec:work_obj} that this is equivalent to the \emph{work of actuation}.  Additionally, we will show that this is a measure of flexibility as it is equivalent to maximizing the displacement of actuation in a particular direction. 

Next, consider the limit of large $\kappa$:
\begin{equation}
	\lim_{\kappa \to + \infty} \ \frac{\kappa \mathcal{C}(\phi, 1) + 1}{\kappa \mathcal{C}(\phi, 0) + 1} = \frac{\mathcal{C}(\phi, 1)}{\mathcal{C}(\phi, 0)}.
\end{equation}
This, for large $\kappa$ the workpiece objective reduces to a ratio of compliances, which we will show is equivalent to the \emph{blocking load objective} in Section~\ref{sec:blocking_obj}. Because this it is a ratio of stimulated to unstimulated compliances, this objective considers not only the actuation flexibility, but also the unstimulated stiffness. 

\subsection{Work of actuation} \label{sec:work_obj}
The work of actuation is the work done against the applied load $f$ as we go from the unactuated to the actuated states:
\begin{equation}
	\label{eq:workobj}
	\mathcal{O}(\Phi) := -\int_\Omega f \cdot ( u_{S=1} - u_{S=0})\, ds
	= \mathcal{C}(\Phi,0) - \mathcal{C}(\Phi,1).
\end{equation}
The identity follows from (\ref{eq:defCompliance}) and shows that the work of actuation is equal to the difference in compliance.  To get further insight into this objective,  consider the case where the modulus of the responsive material is independent of the stimulus. Using~\eqref{eq:responsiveEquilibrium}, the objective~\eqref{eq:workobj} can now be written as
\begin{equation} 
	\label{eq:workobjv}
	\mathcal{O} (\Phi) = - \int_{\partial_f \Omega} f\cdot v \ ds,
\end{equation}
where $v$ satisfies 
\begin{equation} \label{eq:v}
	\begin{cases}
		-\mathrm{div} \left[\chi_s(x)\mathbb{C}_s\varepsilon(v) + \chi_r(x)\mathbb{C}_r \varepsilon(v) \right] =
		-\mathrm{div} \left[\chi_r(x)\mathbb{C}_r \varepsilon^*(1) \right]  & \text{in } \Omega,\\
		v = 0 & \text{on } \partial_u\Omega,\\
		\left[\chi_s(x)\mathbb{C}_s\varepsilon(v) + \chi_r(x)\mathbb{C}_r \varepsilon(v) \right] n = \left[\chi_r(x)\mathbb{C}_r \varepsilon^*(1)\right]  n 
		& \text{on } \partial_f\Omega .
	\end{cases}
\end{equation}
The solution $v$ to (\ref{eq:v}) is the displacement induced in the structure due to only the spontaneous strain field $\chi_r(x)\mathbb{C}_r \varepsilon^*(1)$, and may be expressed as  $v = \Gamma_\Phi \chi_r(x)\mathbb{C}_r \varepsilon^*(1)$ for the appropriate operator $\Gamma_\Phi$.  Thus, our optimal design problem is 
\begin{equation}
	\inf_{\Phi \in {\mathcal D}} \  \int_{\partial_f \Omega} f \cdot \left[\Gamma_\Phi \chi_r(x)\mathbb{C}_r \varepsilon^*(1) \right] \, ds,
\end{equation}
or finding the arrangement of the responsive material that maximizes the resulting spontaneous displacement in the direction of $f$. Thus, the actuation work objective is a measure of flexibility upon stimulation.

\subsection{Blocking load} \label{sec:blocking_obj}

The blocking load is the magnitude of the applied load that nullifies the actuation.  Consider the external traction as scaled by a nonzero constant $\alpha \in \mathbb{R}$, $f = \alpha \bar{f}$, where $\bar{f}$ is some (unit) loading profile.  
The blocking load is the value of $\alpha$ for which the displacement of the actuated structure in the direction of the loading profile vanishes
\begin{equation}
\mathcal{O}(\Phi) := \alpha  \quad \text{where} \quad \mathcal{C}_\alpha(\Phi, 1) = \int_{\partial_f \Omega} \bar{f} \cdot u_{f=\alpha \bar{f}, S=1} ds = 0.
\end{equation}
We now show that this is equivalent to the ratio of compliances.  From (\ref{eq:workobj}) and (\ref{eq:workobjv}),
\begin{equation} \label{eq:c10f}
	\mathcal{C}_\alpha(\Phi, 1)  = \mathcal{C}_\alpha(\Phi, 0) +  \int_{\partial_f \Omega} f \cdot v \, ds = \int_{\partial_f \Omega}  f \cdot (u_{S = 0,\alpha \bar{f}} + v) \, ds
\end{equation}
where $v$ solves (\ref{eq:v}) and is independent of $f$, and $u_{S = 0,\alpha \bar{f}}$ minimizes the elastic energy (\ref{eq:ElastEnergy}) with $S=0$ and $f= \alpha \bar{f}$.  Assuming homogeneous Dirichlet conditions $u_0 = 0$ on $\partial_u \Omega$, it is easy to see using the linearity of the Euler-Lagrange equations that $u_{S = 0,\alpha \bar{f}} = \alpha u_{S = 0,\bar{f}}$.  Therefore, the zero compliance condition $\mathcal{C}_\alpha(\Phi, 1) = 0$ can then be written as
\begin{equation}
	0 = \int_{\partial_f \Omega} f \cdot (\alpha \ \bar{u}_{S = 0,\bar{f}} + v) \, ds = \alpha \int_{\partial_f \Omega} \bar{f} \cdot (\alpha \ \bar{u}_{S = 0,\bar{f}} + v) \, ds,
\end{equation}
or
\begin{equation}
	0 = \int_{\partial_f \Omega} \bar{f} \cdot (\alpha \ \bar{u}_{S = 0,\bar{f}} + v) \, ds .
\end{equation}
The loading amplitude is then
\begin{equation}
	\alpha = -\frac{\int_{\partial_f \Omega} \bar{f} \cdot v \, ds}{\int_{\partial_f \Omega} \bar{f} \cdot \ \bar{u}_{S = 0,\bar{f}} \, ds} 
	= - \frac{\int_{\partial_f \Omega} \bar{f} \cdot \bar{u}_{S = 1,\bar{f}} \, ds}{\int_{\partial_f \Omega} \bar{f} \cdot \ \bar{u}_{S = 0,\bar{f}} \, ds} + 1
	= - \frac{\mathcal{C}_{1}(\Phi, 1)}{\mathcal{C}_{1}(\Phi, 0)} + 1.
\end{equation}
It follows that the blocking load objective is equivalent to minimizing the ratio of the stimulated to unstimulated compliance under fixed load
\begin{equation}
	\inf_{\Phi \in \mathcal{D}} \ \mathcal{O}(\Phi) = \frac{\mathcal{C}(\Phi, 1)}{\mathcal{C}(\Phi, 0)}.
\end{equation}
\vspace{\baselineskip}

We conclude with the comment comparing the two objectives, the work of actuation and the blocking load. Recalling the first identity in (\ref{eq:c10f}), we see that 
\begin{equation}
	\frac{\mathcal{C}_{1}(\Phi, 1)}{\mathcal{C}_{1}(\Phi, 0)} = 1 + \frac{1}{\mathcal{C}_{1}(\Phi, 0)} \int_{\partial_f \Omega} f \cdot v \, ds
\end{equation}
Thus, the blocking load objective is the ratio of the work of actuation objective to the compliance of the unstimulated structure.  Thus, the blocking load objective leads to a structure that balances the work of actuation and the stiffness of the structure.  Finally, these objectives become equivalent when the moduli of the structural and responsive materials are equal, i.e., when ${\mathbb C}_s = {\mathbb C}_r$.  This is because the compliance of the unstimulated state is independent of the design, i.e., ${\mathcal C}(\Phi,0) =C $ is independent of $\Phi$.

\section{Examples of optimal responsive structures} \label{sec:examples}

Here, we explore optimal designs for 2D rectangular lifting actuators. We present results for both the actuation work and blocking load objective, where each are computed under identical computational frameworks.

\begin{figure}
	\centering
	\includegraphics[width=0.3\textwidth]{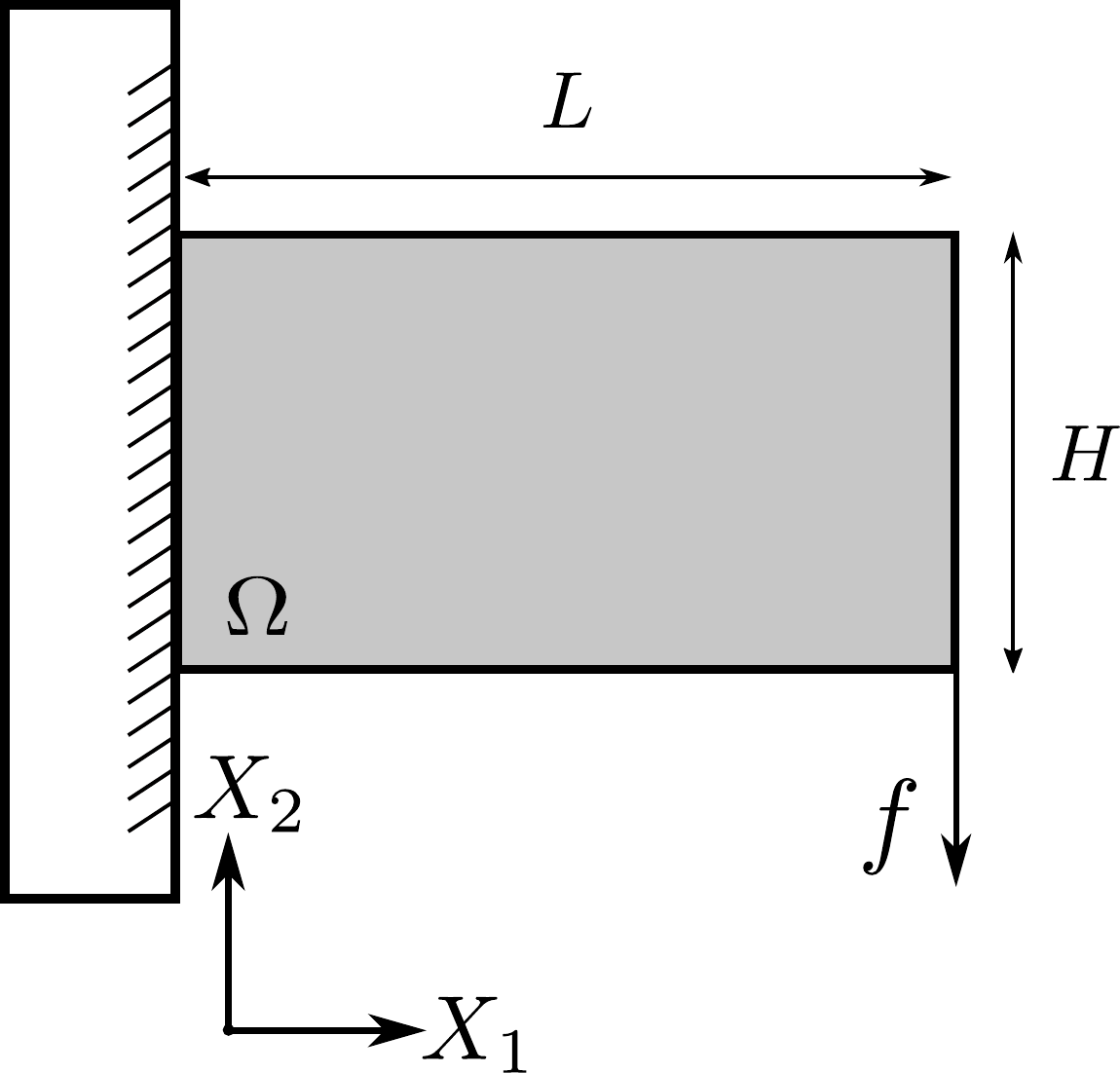}
	\caption{2D Cantilever of length $L$ and height $H$. The left edge at $X_1 = 0$ is fixed rigidly to the wall, with an applied point load $f$ in the bottom right corner.}
	\label{fig:rectangle}
\end{figure}

Consider a 2D rectangular domain $\Omega = (0, L) \times (0, H)$ as shown in Figure~\ref{fig:rectangle}. We take the elastic modulus of both the responsive and structural materials to be isotropic and independent of actuation.  We discretize with standard $p = 1$ Lagrange polynomial shape functions on a quadrilateral mesh through the C++ deal.II finite element library~\cite{Bangerth2007}. The density variable $\phi$ is taken to be constant on each element. As described in the previous section, we use a SIMP interpolation and density filter for regularization. We employ a discrete renormalizing filter as described in~\cite{Bourdin-2001}. This ensures that the density variable is able to take values of $\phi = 1$ near the boundary. Denoting $\phi_k$ as the constant density value of element $k$, and $V(k)$ the set of elements located within distance $R_f$ from element $k$, the filtered density value on element $k$ is
\begin{equation}
	(F * \phi)_{k} = \frac{ \sum_{i \in V(k)} \left( \phi_i \displaystyle \int_{i} F(x - c_{k}) \ dx \right)}{\sum_{i \in V(k)} \displaystyle \int_{i} F(x - c_{k}) \ dx },
\end{equation}
where $c_{k}$ is the center of element $k$. Sensitivities are calculated using the adjoint method, and the density is updated using the method of moving asymptotes (MMA) subject to the linear constraint on total responsive material~\cite{Svanberg1987}. Following convergence of these pixelated designs, a MATLAB$^{\text{\tiny{\textregistered}}}$ code traces smooth contours on the boundaries of the passive and active material domains. We initialize the design to uniform $\phi = \bar{V}_r/V$, and begin iterations thereafter.

\subsection{Optimizing the work of actuation} \label{sec:work}

\begin{figure}
	\centering
	\begin{subfigure}{0.433 \textwidth}
		\centering
		\includegraphics[width=0.99\textwidth]{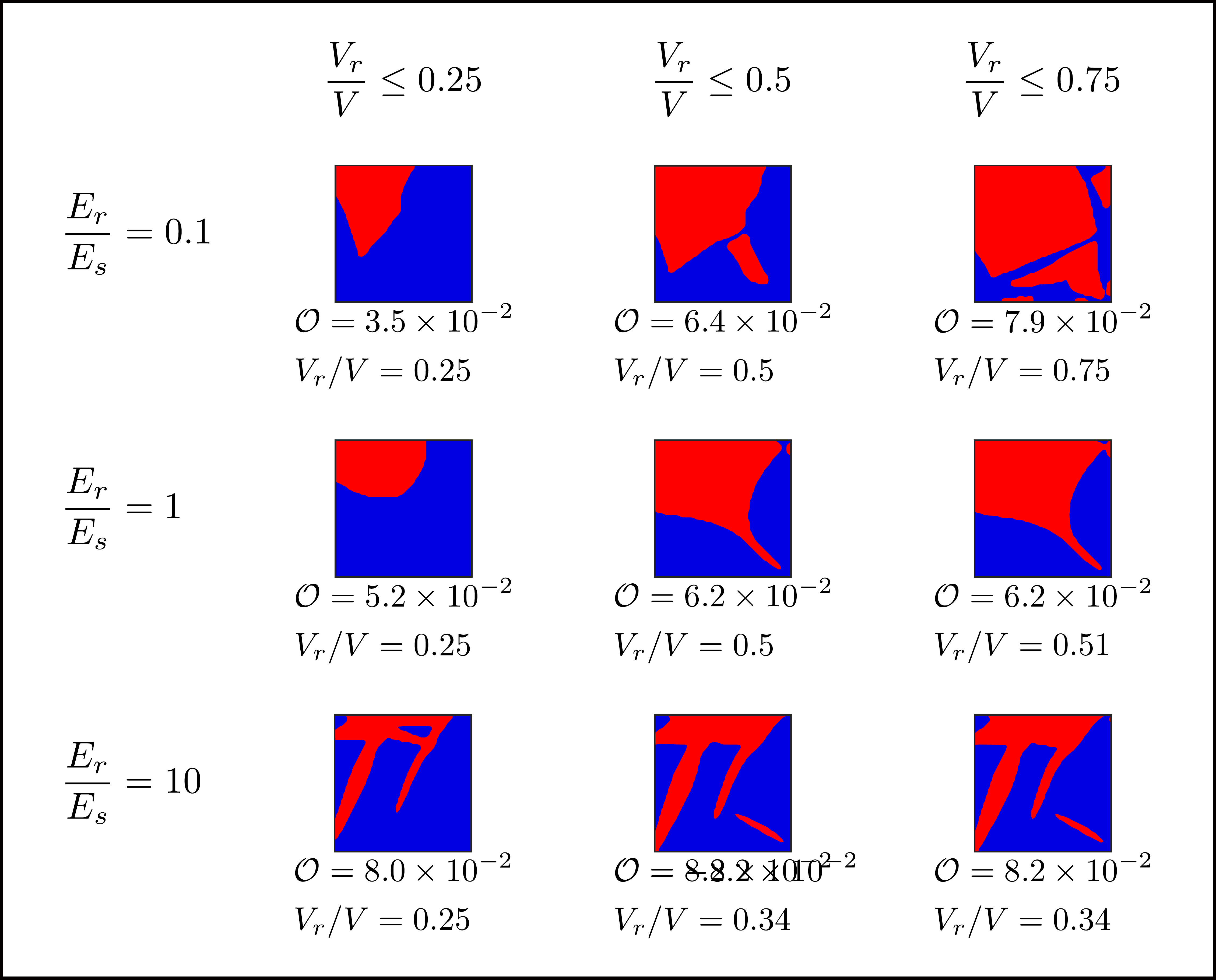}
		\caption{$L/H = 1$}
		\label{fig:bimorph_1}
	\end{subfigure}
	\begin{subfigure}{0.556 \textwidth}
		\centering
		\includegraphics[width=0.99\textwidth]{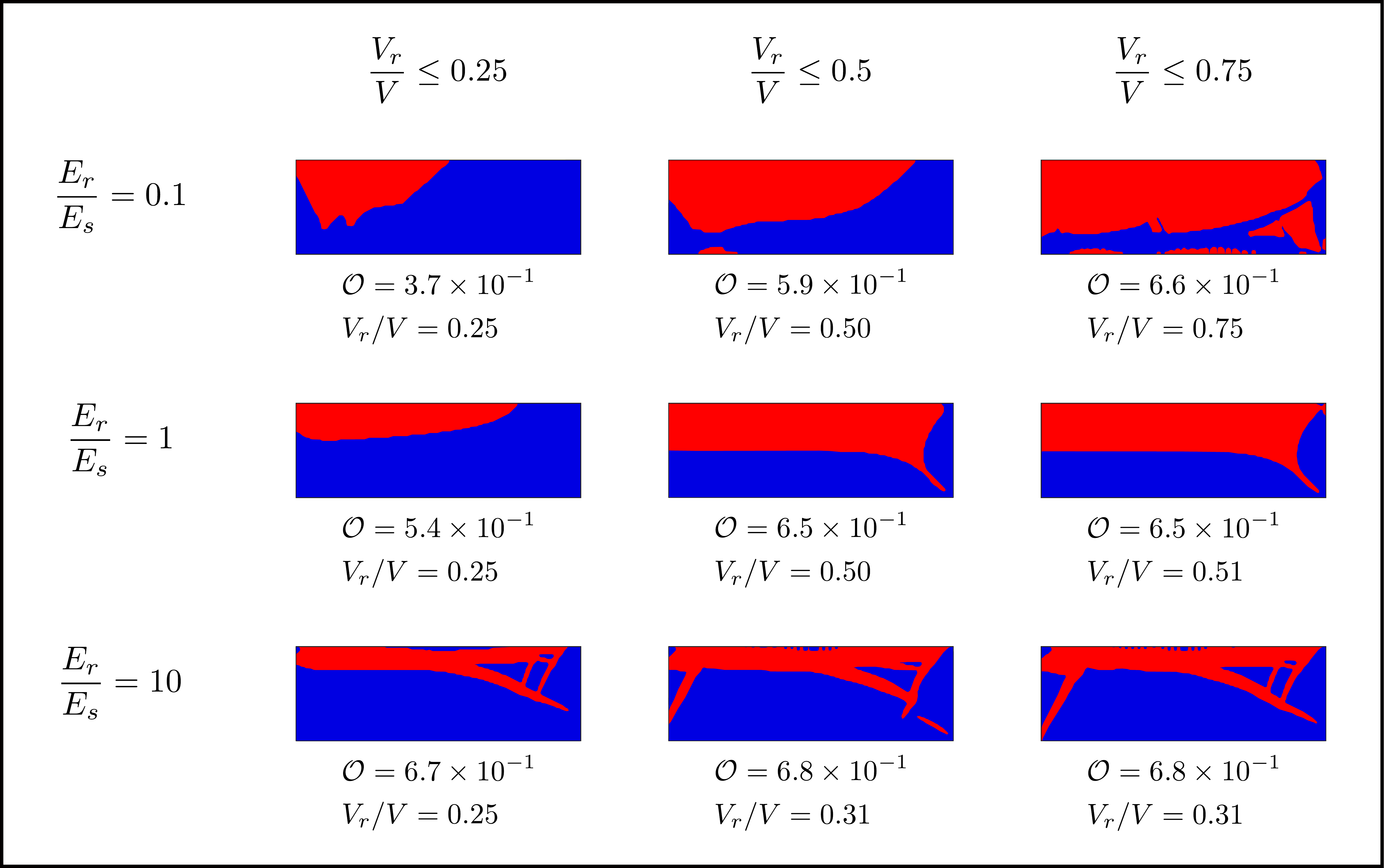}
		\caption{$L/H = 3$} 
		\label{fig:bimorph_3}
	\end{subfigure}
    \par\bigskip
    \begin{subfigure} {0.694 \textwidth}
		\centering
		\includegraphics[width=0.99\textwidth]{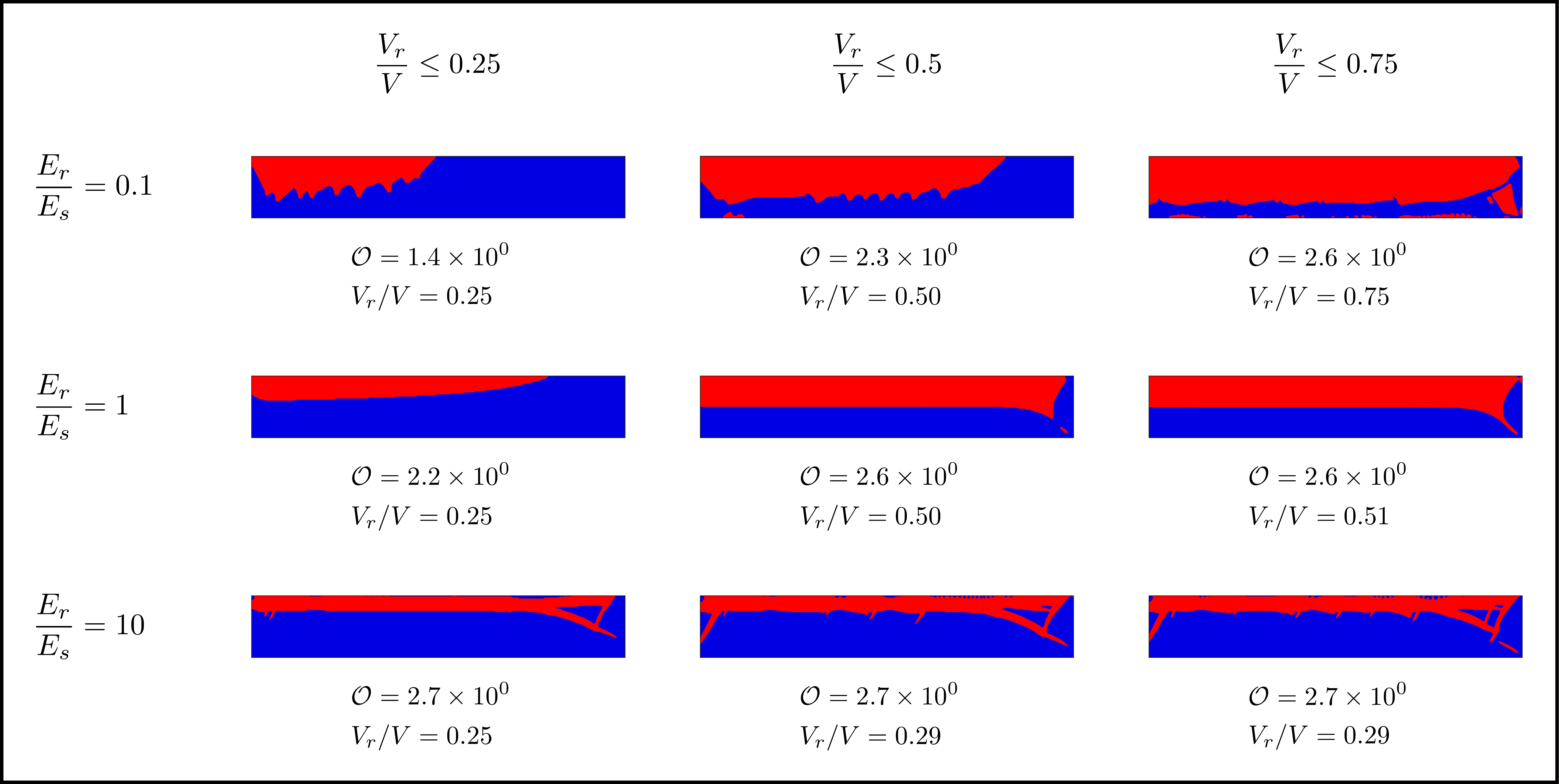}
		\caption{$L/H = 6$}
		\label{fig:bimorph_6}
	\end{subfigure}
	\caption{Converged bimorph designs for optimal work done through actuation. The red and blue regions are the active and passive materials, respectively. Varying aspect ratios, stiffness ratios, and responsive material volume constrains are considered. A Poisson ratio of $\nu = 0.3$ is used for both the passive and responsible materials. Normalized actuation work and volume ratios of the converged designs are shown.}
	\label{fig:work_bimorph_designs}
\end{figure}

We present optimal designs for the lifting actuator optimized for actuation work. Figure~\ref{fig:work_bimorph_designs} shows the converged designs for a spontaneous strain of $\varepsilon^*(1) = -0.1 e_1 \otimes e_1 + 0.1 e_2 \otimes e_2$ (elongation along the horizontal and extension along the vertical) in the responsive material upon stimulation, where $\{e_i\}$ is the standard basis aligned with the axis shown. We investigate these designs for varying domain aspect ratios $L/H$, responsive material volume constraints $\bar{V}_r/V$, and stiffness ratio $E_r/E_s$ of the responsive to passive material. These were computed on uniform finite element meshes of $60 \times 60$, $60 \times 180$, and $60 \times 360$ for aspect ratios of $L/H = 1$, $3$, and $6$, respectively. The filter radius was taken to be $1.5$ times the element width. 

Figure~\ref{fig:work_bimorph_designs} (a), (b) and (c) shows the designs for aspect ratios of $L/H = 1$, $3$, and $6$ respectively.  For each aspect ratio, the rows show the designs for fixed ratio of elastic moduli $E_r/E_s$ as the allowable ratio of active to passive material $V_r/V$ is varied.  Similarly, the columns show the designs for fixed allowable ratio of active to passive material $V_r/V$ but varying ratio of elastic moduli $E_r/E_s$.  Note that the ratio of responsive to passive material $V_r/V$ is imposed as an inequality constraint: the actual ratio used is indicated below each design.  Finally, the optimal value of the objective is also indicated below each design.

To understand these results, we start with Figure~\ref{fig:work_bimorph_designs} (b) where $L/H = 3$, and specifically with the middle row where $E_r/E_s=1$.  The design is similar to the bi-material strip with the active material on the top and passive material on the bottom.  When stimulated, the active material contracts horizontally, causing the domain to bend upward and perform work against the load.   When the allowed volume fraction of responsive material is small (left), the design uses the entire allowed volume fraction and places it close to the support since it can provide the maximum moment against the load.  As the allowed volume fraction of responsive material increases, the design continues to use the entire allowed volume fraction with roughly a uniform thickness.  However, at large allowed volume fraction (right), the design does not use the full allocation.  Instead, it saturates at about $51\%$ because it needs a sufficient amount of passive material to convert its horizontal contraction into work against the vertical load.   The value of the objective increases with the allowed volume fraction of responsive material, but saturates when the volume fraction does.

The designs remain roughly similar as we change the ratio of the stiffness of the responsive material to that of the structural material ($E_r/E_s$). The design uses more responsive material when it is more compliant (top row of  Figure~\ref{fig:work_bimorph_designs} (b) ), as it requires more of the responsive material to actuate against the stiffer structural material.  The opposite is true when the responsive material is stiffer (bottom row).  The value of the objective increases with the relative stiffness of the responsive material when we fix the allowed volume fraction (columns); however, the saturated value when we allow sufficient volume fraction is relatively independent of the stiffness ratio.

We now study the effect of aspect ratio $L/H$ comparing the designs of Figure~\ref{fig:work_bimorph_designs} (b) with those in Figure~\ref{fig:work_bimorph_designs} (a, c).  The designs and the trends against allowable  volume fraction of responsive material and stiffness ratio are similar (except for high stiffness of the responsive material and short aspect ratio where the design has diagonal laminates to provide stiffness against shear).  The optimal value increases with aspect ratio.

%

\subsection{Optimizing the blocking load} \label{sec:blocking}

We consider the same domain and loading as described in Figure~\ref{fig:rectangle}, and look to optimize the blocking load applied to the bottom right corner. The numerical schemes are identical, with the only difference being the objective function. Figure~\ref{fig:blocking_bimorph_designs} shows the converged designs for a spontaneous strain of $\varepsilon^*(1) = -0.1 e_1 \otimes e_1 + 0.1 e_2 \otimes e_2$ in the responsive material upon stimulation. The designs and the trends are broadly similar to those obtained by optimizing the work of actuation.  When the stiffnesses of the two materials are the same ($E_r/E_s = 1$) the designs coincide since the objectives are identical as noted above.  In the other situations, the blocking load designs tend to use more stiffer material (more structural material when $E_r/E_s = 0.1$ and more responsive material when $E_r/E_s = 10$).  We also see more diagonal reinforcement.

\begin{figure}
	\centering
	\begin{subfigure}{0.433 \textwidth}
		\centering
		\includegraphics[width=0.99\textwidth]{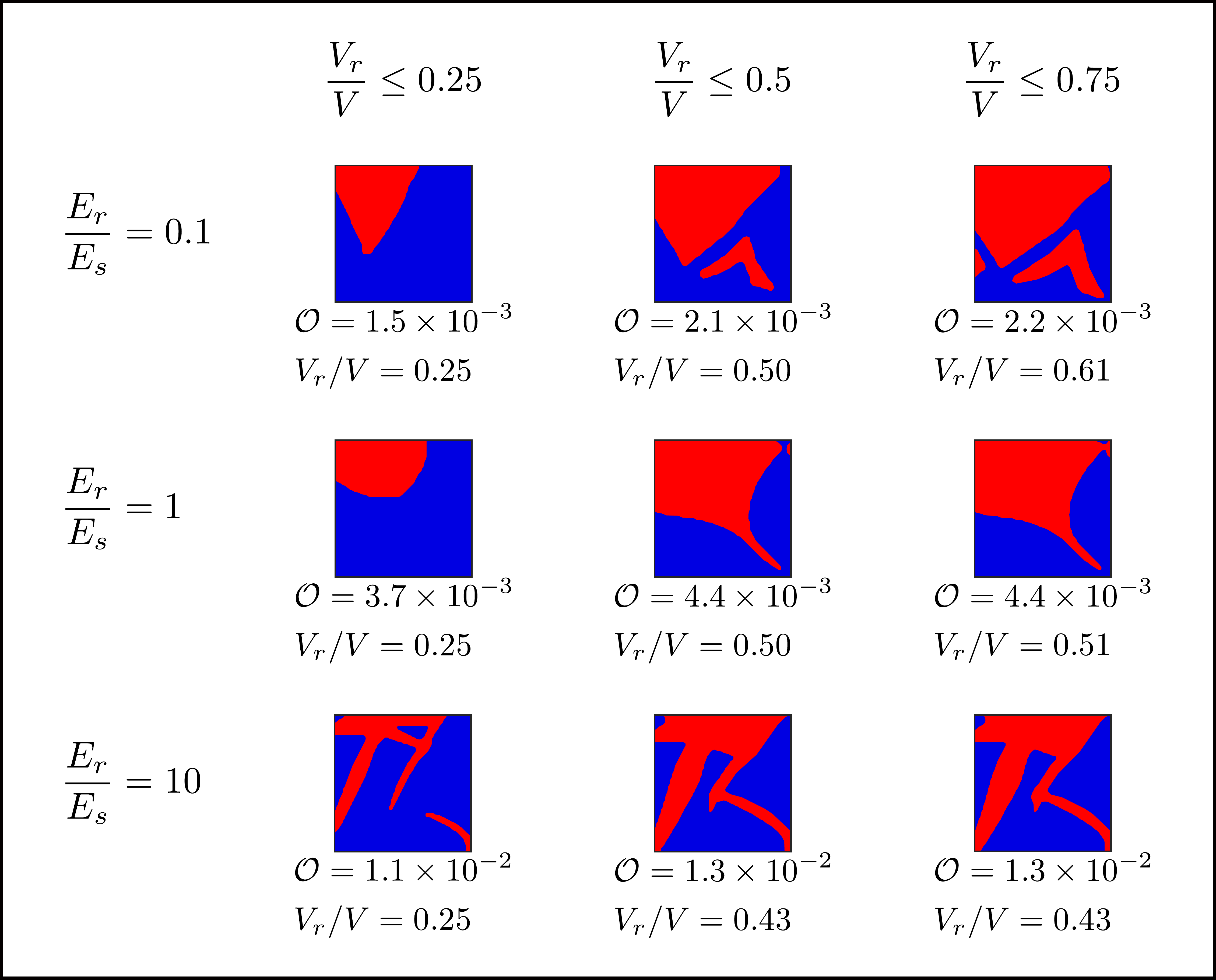}
		\caption{$L/H = 1$}
		\label{fig:bimorph_blocking_1}
	\end{subfigure}
	\begin{subfigure}{0.556 \textwidth}
		\centering
		\includegraphics[width=0.99\textwidth]{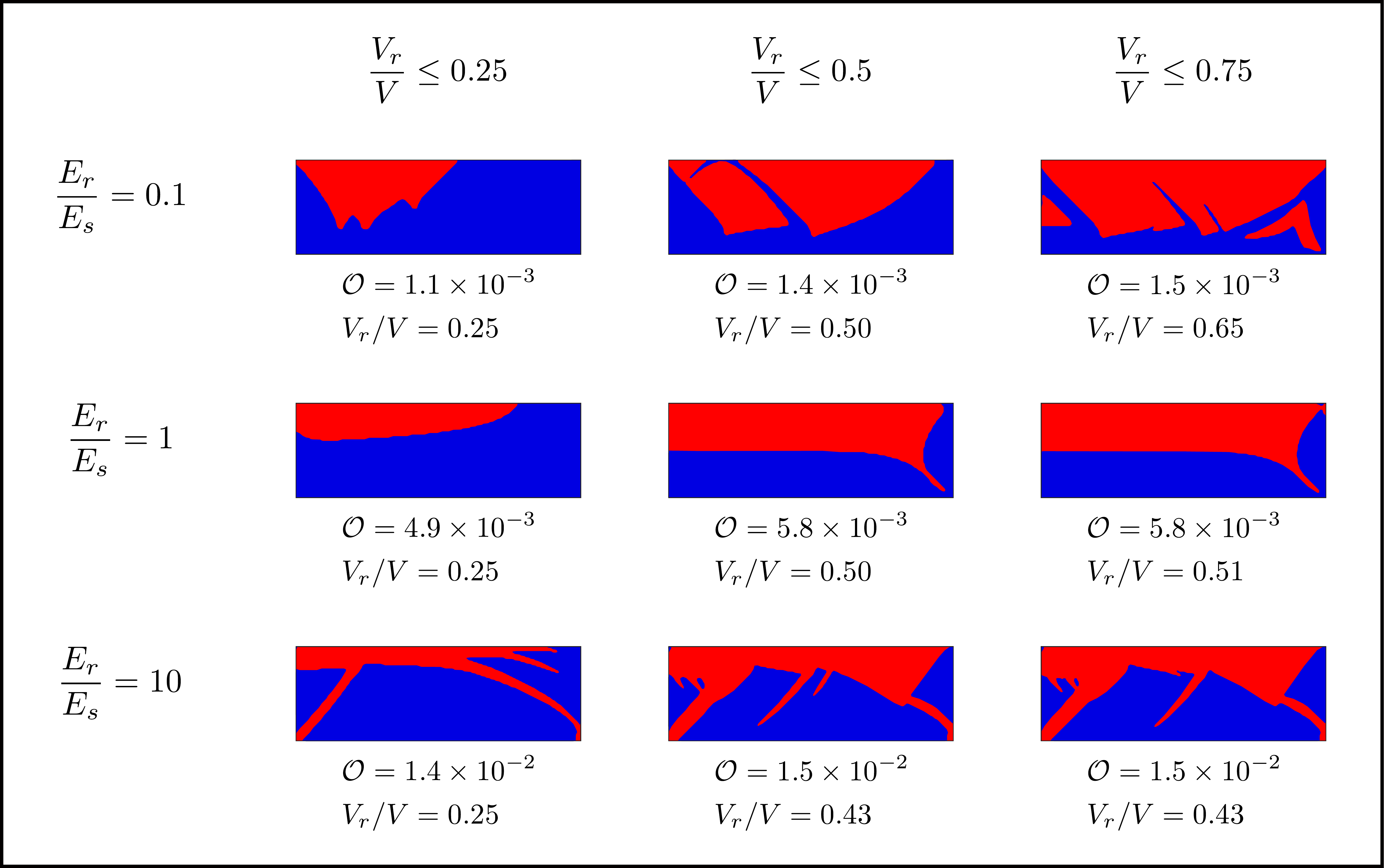}
		\caption{$L/H = 3$}
		\label{fig:bimorph_blocking_3}
	\end{subfigure}
	\par\bigskip
	\begin{subfigure} {0.694 \textwidth}
		\centering
		\includegraphics[width=0.99\textwidth]{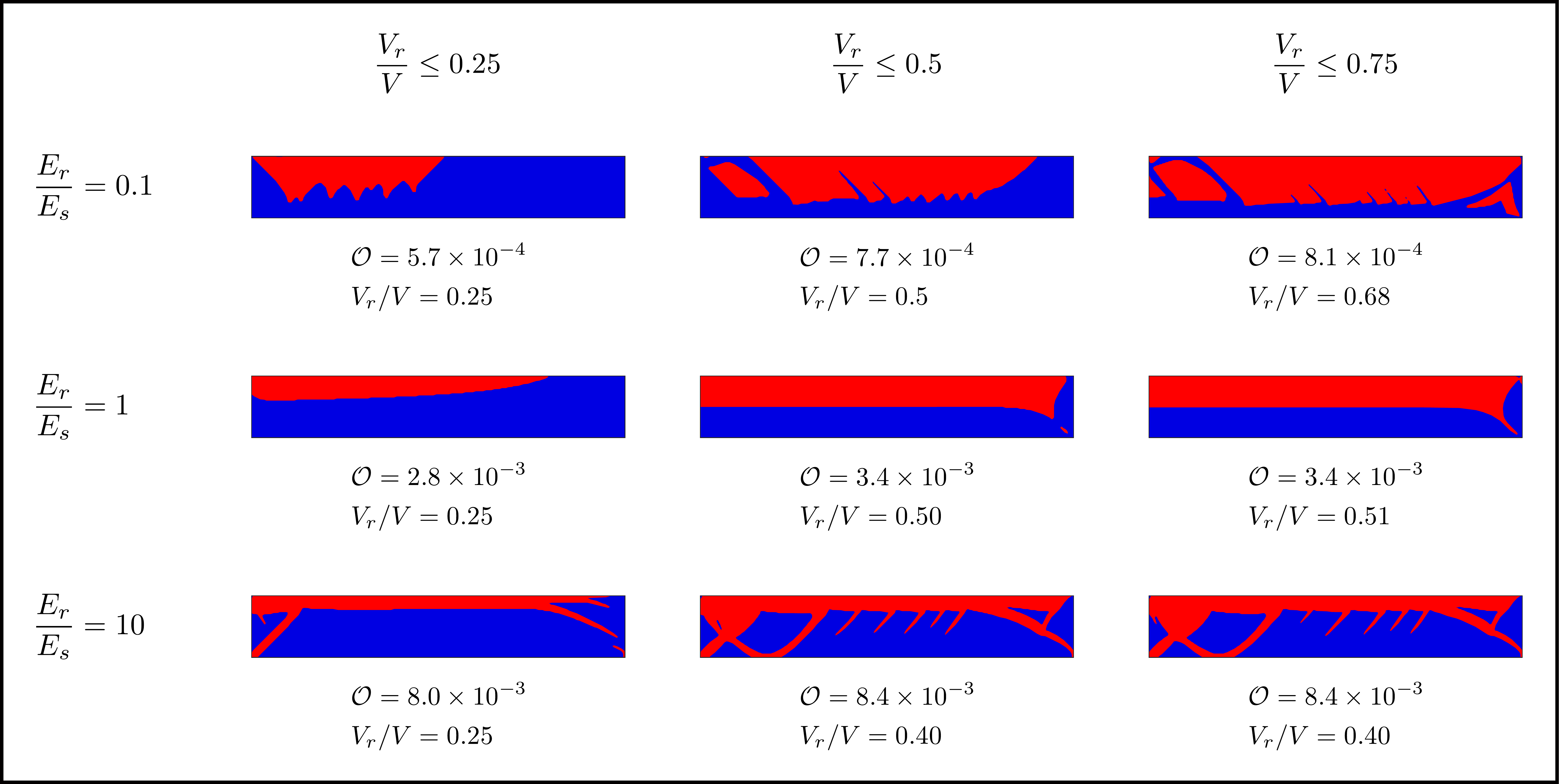}
		\caption{$L/H = 6$}
		\label{fig:bimorph_blocking_6}
	\end{subfigure}
	\caption{Converged bimorph designs for optimal actuation blocking load. The red and blue regions are the active and passive materials, respectively. Varying aspect ratios, stiffness ratios, and responsive material volume constrains are considered. A Poisson ratio of $\nu = 0.3$ is used for both the passive and responsible materials. Normalized blocking load values and volume ratios of the converged designs are shown. }
	\label{fig:blocking_bimorph_designs}
\end{figure}

\section{Optimizing responsive structures with voids} \label{sec:voids}

Motivated by a structural frame actuated by muscle-like actuators, we now consider a responsive structure with voids or holes in the domain.  We now have  $\chi_s + \chi_r \le 1$.  We introduce a SIMP interpolation and filter as before through an additional density variable. We consider $\rho: \Omega \mapsto [\rho_{min}, 1]$ for some $1 > > \rho_{min} >0$, which determines void or solid and consider the energy functional
\begin{multline} \label{eq:voidenergy}
\mathcal{E}_v(\phi, \rho, u, S) := \int_{\Omega} (F * \rho)^p \left( \frac{1}{2}\right)\big[  (1 - (F * \phi)^p) \mathbb{C}_s \varepsilon(u) \cdot \varepsilon(u) \\
+  (F * \phi)^p \mathbb{C}_r(S)(\varepsilon(u) - \varepsilon^*(S)) \cdot (\varepsilon(u) - \varepsilon^*(S)) \big] \ dx 
- \int_{\partial_f \Omega} f \cdot u \ ds.
\end{multline}
The compliance is, again, 
\begin{equation} \label{eq:compliancevoid}
\mathcal{C}(\phi, \rho, S) = \int_{\partial_f \Omega} f \cdot u \ ds,
\end{equation}
where $u$ is an equilibrium solution,
\begin{equation} 
u = \arg \min_{u \in \mathcal{U}} \ \mathcal{E}_v(\phi, \rho, u, S).
\end{equation}
Note that the voids have some residual stiffness since $\rho_{min} > 0$ to maintain the coercivity of (\ref{eq:voidenergy}).  We choose this small enough so that it has only a limited effect on the resulting designs. We again consider a compliance dependent objective
\begin{equation} \label{eq:optimization_problem_filter_void}
\inf_{\phi \in \mathcal{D}_f, \ \rho \in \mathcal{R}_f} \ \mathcal{O}(\phi, \rho) := \bar{\mathcal{O}}(\mathcal{C}(\phi, \rho, S_1), \mathcal{C}(\phi, \rho, S_2) ),
\end{equation}
where we optimize over the space of feasible designs
\begin{equation} \label{eq:RfDf}
\begin{aligned}
\mathcal{D}_f &= \left\{ \phi : \phi \in [0, 1] \text{ a.e. in } \Omega, \int_{\Omega} \rho \phi \ dx \leq \bar{V}_r \right \}, \\
\mathcal{R}_f &= \left\{ \rho : \rho \in [\rho_{min}, 1] \text{ a.e. in } \Omega, \int_{\Omega} \rho \ dx \leq \bar{V}_0 \right \}, 
\end{aligned}
\end{equation}
where $\bar{V}_0$ and $\bar{V}_r$ are the allowed volumes total material and responsive material, respectively.
\begin{theorem} \label{thm:ext_void}
	Recall the definition of the compliances from~\eqref{eq:compliancevoid}, and set
	\begin{equation}
	\mathcal{O}(\phi, \rho) := \bar{\mathcal{O}}(\mathcal{C}(\phi, \rho, S_1), \mathcal{C}(\phi, \rho, S_2) ),
	\end{equation}
	where $\bar{\mathcal{O}}$ is bounded below and continuous. There exists a $\bar{\phi} \in \mathcal{D}_f$ and $\bar{\rho} \in \mathcal{R}_f$ such that,
	\begin{equation} \label{eq:O_inf_void}
	\mathcal{O}(\bar{\phi}, \bar{\rho}) = \inf_{\phi \in \mathcal{D}_f, \ \rho \in \mathcal{R}_f } \ \mathcal{O}(\phi, \rho).
	\end{equation}
\end{theorem}

\begin{proof}
    The weak continuity results from Lemmas \ref{lem:00} and \ref{lem:01} can be extended for the additional filtered density field $\rho$. The rest of the proof follows the same steps as the proof for Theorem~\ref{thm:ext}. 
\end{proof}

We now consider the two objectives that we introduced in the previous section.  We begin with the work of actuation in Section \ref{sec:work} which is the difference between the compliances in the stimulated and unstimulated states.  However, the compliances are not bounded since we have voids\footnote{Precisely, it is bounded by a constant that depends on $\rho_{min}$ and becomes unbounded as $\rho_{min} \to 0$}.  Thus, this objective does not satisfy the hypothesis of the theorem, and a brute-force implementation does not converge to meaningful designs.

So we focus on the blocking load or mechanical advantage introduced in Section \ref{sec:blocking}.  Since this objective considers the ratio of the two compliances, it remains bounded satisfying the hypothesis of the theorem above.  Specifically, we consider the optimization problem
\begin{equation}
\inf_{\phi \in \mathcal{D}_f,\  \rho \in \mathcal{R}_f} \ \mathcal{O}(\phi) = \frac{\mathcal{C}(\phi, \rho, 1)}{\mathcal{C}(\phi, \rho, 0)}.
\end{equation}

The numerical schemes are nearly identical to the case of no holes, except for an additional density field. We consider this density variable constant on each element. We adopt a sequential update scheme to handle the nonlinear constraint posed in~\eqref{eq:RfDf}. After obtaining sensitivities through the adjoint method, the discrete $\rho$'s are updated using MMA under the linear constraint of allowable material. Then, using the newly updated $\rho$'s to write the constraint as linear, we update the $\phi$'s with another MMA. This results in only applying linear constraints for updates.

\subsection{Example in two dimensions: Lifting actuator} \label{sec:voids2D}

\begin{figure}
	\centering
	\includegraphics[width=0.9\textwidth]{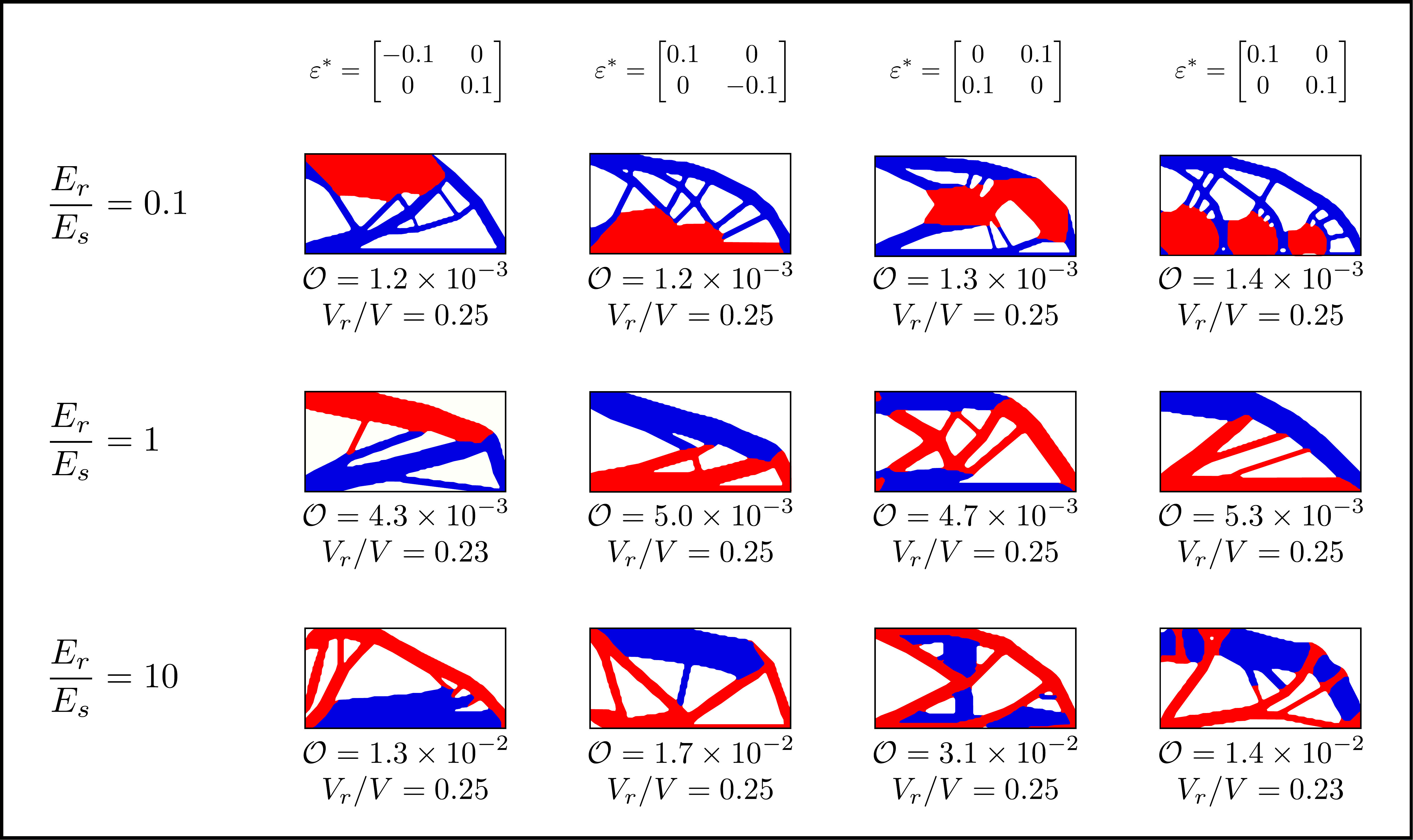}
	\caption{Converged designs for maximum blocking load of a 2D cantilever structure with aspect ratio $L/H = 2$. The inequality constraints $V_0/V \leq 0.5$ and $V_r/V \leq 0.25$ are enforced for all cases. The red and blue regions are the active and passive materials, respectively. Designs are shown for varying stiffness ratios for different spontaneous strains. A Poisson ratio of $\nu = 0.3$ is used for both the passive and responsible materials. Normalized blocking load values and converged responsive material volume ratios are shown. In all cases, the designs converged to $V_0/V = 0.5$. }
	\label{fig:multi_blocking}
\end{figure}

We look to optimize the blocking load applied to the bottom right corner of Figure~\ref{fig:rectangle}. We consider a domain aspect ratio of $L/H = 2$, and a uniform finite element mesh of $60\times120$ quadrilateral elements. The filter radius for both of the SIMP variables is taken as $1.5$ times the element width, $R_f = 0.0125 L$.  We constrain the responsive materials to be a quarter of less of the total area ($V_r/V \leq 0.25$), and the combination of the structural and responsive materials to be less than half the total area ($V_0/V \leq 0.5$ ). The designs are all initialized to uniform $\phi = \bar{V}_r/\bar{V}_0$ and $\rho = \bar{V}_0/V$.

Figure~\ref{fig:multi_blocking} shows converged designs following contour smoothing in MATLAB$^{\text{\tiny{\textregistered}}}$. Designs are shown for various spontaneous strains of responsive material and various ratios of the stiffnesses of the responsive and structural materials: the columns have the same spontaneous strain while the rows have the same stiffness ratio.  Consider the first column where the spontaneous strain contracts along the horizontal and expands along the vertical.  As in the situation without the voids (Figure \ref{fig:work_bimorph_designs}), the active material is concentrated on the top. Also as before, in the cases of large stiffness contrast we see thick domains of the softer material, whether that be passive or responsive.  However, in this situation,  the stiffer material resembles a frame as in the classical problem of optimizing the compliance under a volume constraint.

The overall shape remains similar even when the spontaneous strain of responsive materials change.  However, the placement of the responsive material changes significantly.   For example, in the second column where the spontaneous strain is an elongation along the horizontal direction and contraction along the vertical, the responsive material is concentrated at the bottom. It should be noted that in all of these cases the designs saturated the total allowed material, converging to $V_0/V = 0.5$. Additionally, nearly all of the designs saturated the constraint on responsive material.

\subsection{Example in three dimensions: Torsional actuator} \label{sec:voids3D}

Thus far we have only considered the design of plane strain 2D structures. Here, we study the 3D design of a torsional actuator for optimal blocking load. We consider the cylindrical domain shown in Figure~\ref{fig:cylinder}, with one face completely fixed to a wall and uniform tangential loading applied to its far edge. Optimizing the blocking load under this loading is analogous to maximizing the blocking torque of the actuator. We consider the formulation of the previous section for two materials with voids, with identical numerical schemes.
\begin{figure}
	\centering
	\includegraphics[width=0.4\textwidth]{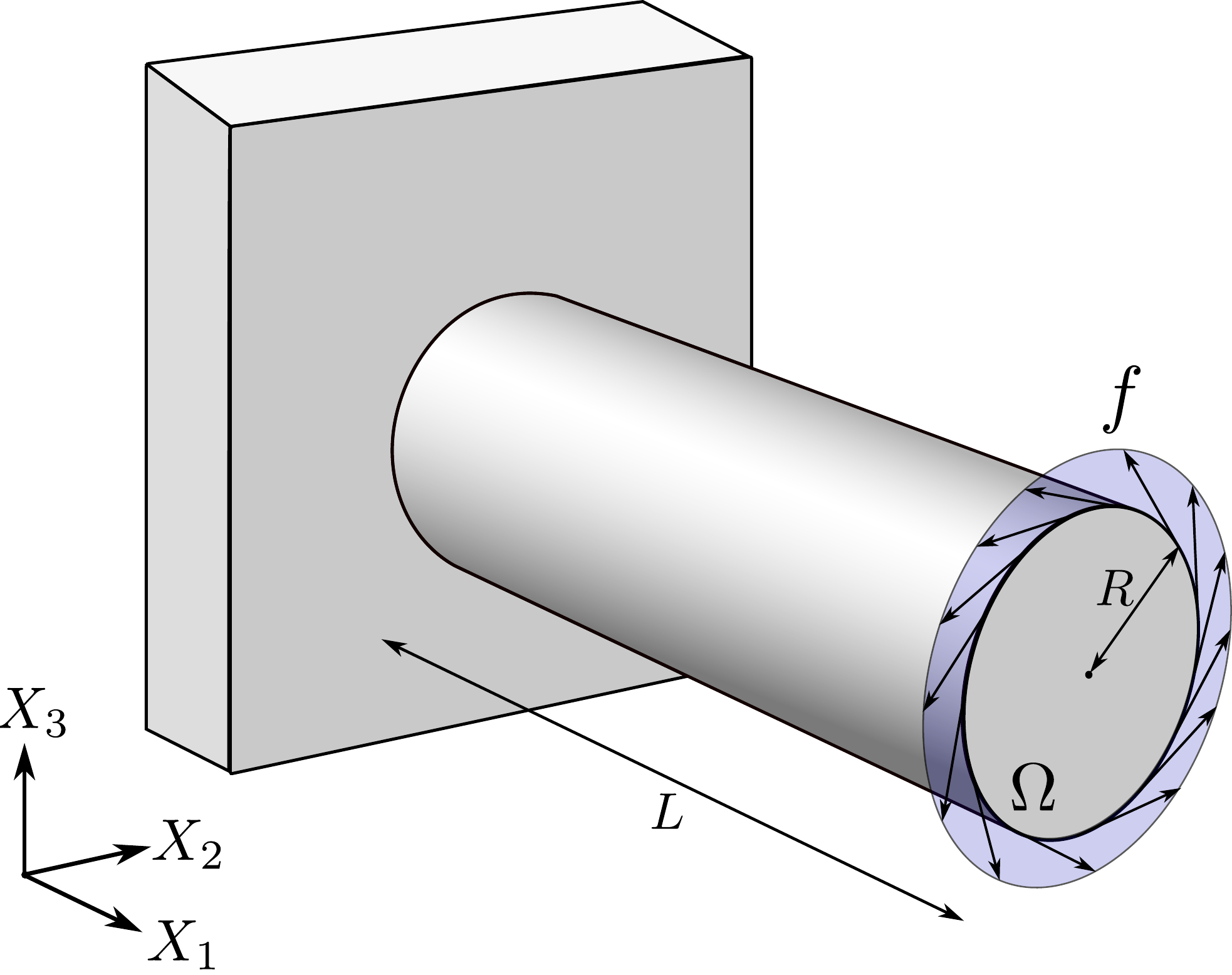}
	\caption{Cylindrical geometry for the torsional actuator of radius $R$ and length $L$. The face at $X_1 = 0$ is fixed rigidly to the wall, and the far edge at $X_1 = L$ has uniform tangential loading. }
	\label{fig:cylinder}
\end{figure}

We consider a cylindrical domain of aspect ratio $L/R = 4$, computed on a finite element mesh of 245,760 hexahedral elements. This corresponds to a characteristic element side length of about $L/100$. To account for the variability in element sizes, the filter radius was taken rather large at 3 times this length or $R_f = 0.03 L$.  We investigate varying total material volume constraints $V_0/V$, and stiffness ratios $E_r/E_s$ of the responsive to passive material. We consider the volume of responsive material to be constrained to $V_r / V_0 \leq 0.5$. Thus, we restrict the amount of responsive material to be less than half the amount of total material in the domain. We consider two cases of spontaneous strain upon stimulation: a transversely isotropic elongation along $X_1$ and contractions in the $X_2 - X_3$ plane, as well as isotropic contraction. 

Figure~\ref{fig:cylinder_transverse} shows converged designs of the torsional actuator for a transversely isotropic transformation strain of $\varepsilon^*(1) = 0.1 e_1 \otimes e_1 - 0.05 e_2 \otimes e_2 - 0.05 e_3 \otimes e_3$. As one would expect, we see the responsive material is arranged helically towards the outer edges of the domain. Similarly to the 2D case, for $E_r/E_s = 0.1$ we see thicker clumps of responsive material, where for $E_r/E_s = 10$ it is spread more thinly. For $V_0/V = 0.25$, the material is mostly concentrated towards the outer edges of the domain, with more material near the center for $V_0/V = 0.5$. This is understandable, as torsional stiffness is maximized by placing material farther from the center. 

Figure~\ref{fig:cylinder_isotropic} shows converged designs for the torsional actuator for an isotropic transformation strain of $\varepsilon^*(1) = -0.033\ I_{3\times 3}$. The general trends outlined in the previous discussion remain valid. However, the direction of the helical responsive material now goes in the opposite direction, as the responsive material now contracts rather than expands along the $X_1$ direction. 

For both cases of spontaneous strains, the designs completely saturate the total allowed material converging to $V_0 = \bar{V}_0$. Additionally, the volume of responsive material is nearly saturated in all cases. 

In the previous 2D cases, the gradients of the objective with the design variables for the initial uniform density designs were sizeable. This lead to fast convergence through the MMA algorithm from initialization. However, for the 3D torsional actuator, the gradients of the blocking load objective with the design variable $\phi$ were nearly zero for a uniform density design, especially in the case of $E_r/E_s = 1$. This resulted in dozens of early iterations with small changes to design. To remedy this, we consider an initial nonuniform configuration
\begin{equation}
    \phi(X_1, X_2, X_3) = (V_r/V) + \epsilon \cos \left(2 \theta - \frac{2 \pi X_1}{L} \right),
\end{equation}
where $\theta = \tan^{-1}(X_3/X_2)$. $\epsilon = 0.05$ was used for the previously described designs. This initial design is a small perturbation towards a helical $\phi$ with two ``strands'' running along the $X_1$ axis. This is the reason that the converged designs all have two main ``strands'' of responsive material. While the form of the perturbation may seem presumptuous, the magnitude of the perturbation was small. We also used the same perturbation for both cases of spontaneous strain which resulted in converged designs with helices in different directions. Additionally, initializing the designs with random perturbations resulted in designs that, while different, had objective values within $3\%$ of that of the helical perturbation. Thus, we argue that this perturbation is an acceptable means to quicker convergence. It should be noted that while we have proven existence of solutions, there is no uniqueness. It is expected for problems of this nature to have multiple local minima and for the initial guess to have a sizeable effect on the converged design. However, so long as the final objective value does not differ markedly, the designs are all adequate.

\begin{figure}[ht]
	\centering
	\includegraphics[width=0.92\textwidth]{figures/cylinders_transverse_obj_vol.pdf}
	\caption{Converged designs for maximum blocking torque on the cylindrical domain for the transversly isotropic transformation strain $\varepsilon^*(1) = 0.1 e_1 \otimes e_1 - 0.05 e_2 \otimes e_2 - 0.05 e_3 \otimes e_3$. The red is the responsive material and the blue passive. Designs are shown for varying moduli ratios and amount of total allowed material. The ratio of responsive material to passive material was constrained to $V_r/V_0 \leq 0.5$ for all cases. That is, the left column is constrained to $V_r/V \leq 0.125$ and the right to $V_r/V \leq 0.25$. Normalized blocking torque values and converged responsive material volume ratios are shown. In all cases, the designs converged to $V_0/V = 0.5$.}
	\label{fig:cylinder_transverse}
\end{figure}

\begin{figure}[ht]
	\centering
	\includegraphics[width=0.92\textwidth]{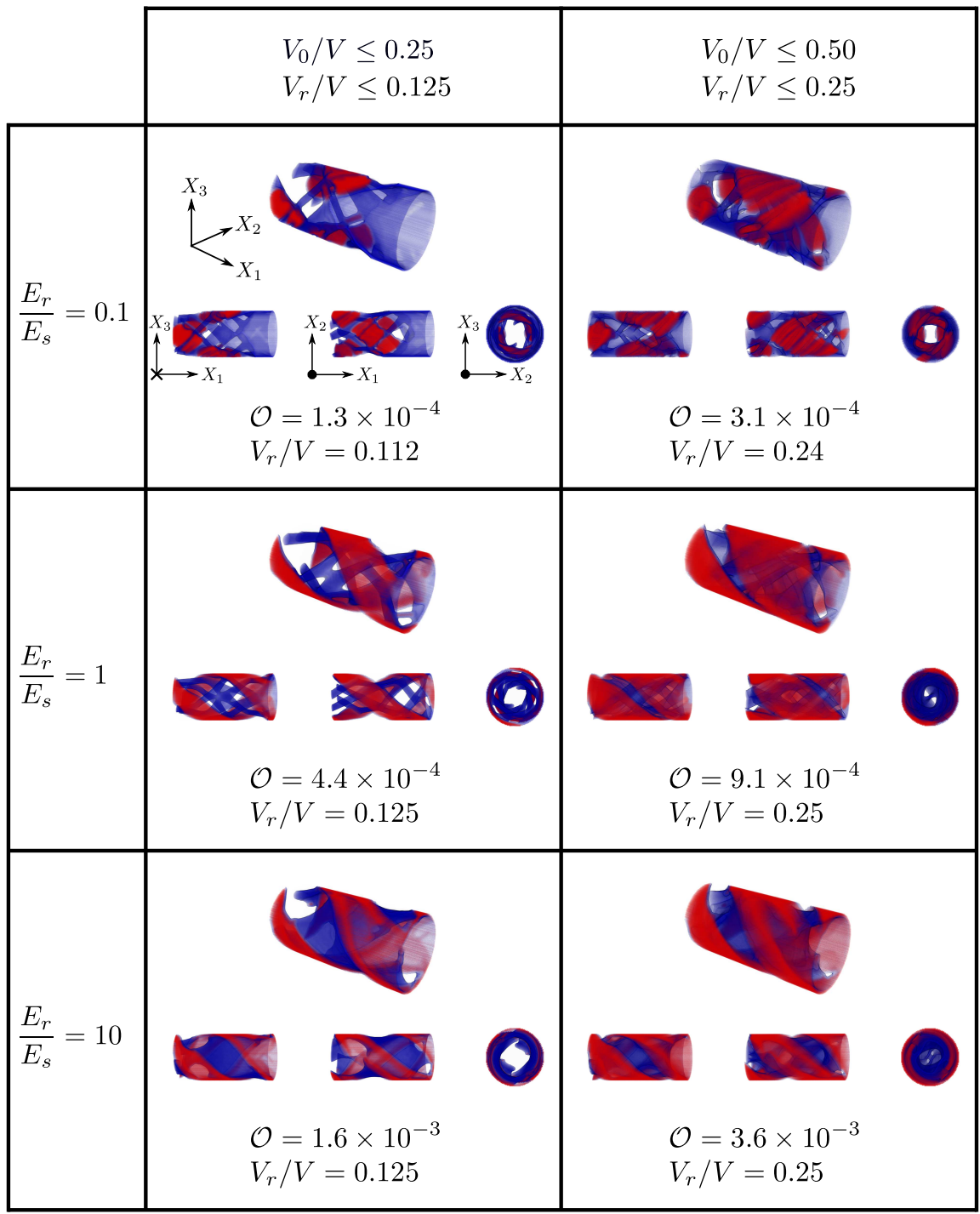}
	\caption{Converged designs for maximum blocking torque on the cylindrical domain for the volumetric transformation strain $\varepsilon^*(1) = -0.033 \ I_{3\times 3}$. The red is the responsive material and the blue passive. Designs are shown for varying moduli ratios and amount of total allowed material. The ratio of responsive material to passive material was constrained to $V_r/V_0 \leq 0.5$ for all cases. That is, the left column is constrained to $V_r/V \leq 0.125$ and the right to $V_r/V \leq 0.25$. Normalized blocking torque values and converged responsive material volume ratios are shown. In all cases, the designs converged to $V_0/V = 0.5$.}
	\label{fig:cylinder_isotropic}
\end{figure}

\section{Conclusions} \label{sec:con}

We have investigated the optimal design of responsive structures through topology optimization. By considering a filtering scheme and SIMP interpolation, we have proven existence of optimal designs for a class of objective functions dependent on the compliances of the stimulated and unstimulated states. In particular, we have considered the actuation work and blocking load objective. We showed that these can both be written as functions of compliances. For each of these objectives, we presented  numerical results for the design of bimorph actuators on a 2D rectangular domain. The converged designs contain complex structures that would otherwise be difficult to intruitively conjure, especially for the blocking load objective. Additionally, we considered the introduction of voids for the blocking load optimal design. This resulted in a rich array of structures highly dependent on the spontaneous strain and the stiffness ratios of the passive and responsive materials. Finally, we considered the design of a 3D torsional actuator for maximum blocking torque. We investigated the design for varying stiffness and volume ratios for two cases of transformation strain. As expected, the converged designs had responsive material distributed in helices at the outer edges of the domain, with the direction dependent on the transformation strain. 

We now discuss further directions that could extend this work. Here, we only consider linear elastic materials. As many active materials may undergo finite strains through both deformation and actuation, it may be worth investigating the design of structures under richer material models. In particular, geometric nonlinearities may lead to insightful designs. Another extension is coupling the stimulation and response. Physically, this could be realized though a number of mechanisms including magnetostriction, heat diffusion for shape memory alloys, or photo-responsive materials. The challenges would not only include the formulation and implementation for such a system, but also the choice of a suitable objective function. Additionally, in our work we considered the actuation strain to be prescribed and constant throughout the whole domain. With the recent developments in directional 3D printing in materials such as liquid crystal elastomers~\cite{ketal_advmat_19}, an interesting extension might involve optimizing over the responsive materials spatially varying direction as well.

\FloatBarrier

\paragraph{Acknowledgement}  We are grateful for the financial support of the U.S. National Science Foundation through ``Collaborative Research: Optimal Design of Responsive Materials and Structures'' (DMS:2009289 at Caltech and DMS:2009303 at LSU and McMaster University).


\bibliographystyle{abbrv}
\bibliography{responsive}

\newpage
\appendix

\section{Workpiece objective as force in spring} \label{ap:workpiece}
Here we show the workpiece objective is equivalent to maximizing the load of a point spring. Consider a linear spring in direction $\hat{n}$ of spring constant $\kappa > 0$ connected to the boundary of the domain at some point of interest $x_0 \in \partial_f \Omega$. The aim is to maximize the load carried by this spring upon actuation. Thus, we look to maximize the load in the spring:
\begin{equation}
	\sup \{ f_0 : f_0 = -\kappa u(x_0) \cdot \hat{n},  \Phi \in \mathcal{D} \}
\end{equation}
where $u$ is the equilibrium solution corresponding to $S = 1$ and $f = f_0 \delta(x - x_0) \hat{n}$. Assuming homogeneous Dirichlet conditions $u_0 = 0$ on $\partial_u \Omega$, it is easy to see using the linearity of the Euler-Lagrange equations
\begin{equation}
	u = v + u_{S = 0, f_0 \hat{n}}	
\end{equation}
where $u_{S = 0, f_0 \hat{n}}$ minimizes the elastic energy~\eqref{eq:ElastEnergy} with $S = 0$ and $f = f_0 \delta(x - x_0) \hat{n}$. Invoking linearity again gives $u_{S = 0, f_0 \hat{n}} = f_0 u_{S = 0, \hat{n}} $. The displacement can then be written as
\begin{equation}
	u = v + f_0 u_{S = 0, \hat{n}}.
\end{equation}
Evaluating at $x = x_0$, taking an inner product with $\hat{n}$, and using the constraint that $f_0 = - \kappa u(x_0) \cdot \hat{n}$ gives
\begin{equation}
	-\frac{f_0}{\kappa} = v(x_0) \cdot \hat{n} + f_0 u_{S = 0, \hat{n}}(x_0) \cdot \hat{n}.
\end{equation}
Rearranging gives
\begin{equation}
	f_0 = \frac{-\kappa v(x_0) \cdot \hat{n} }{\kappa u_{S = 0, \hat{n}}(x_0) \cdot \hat{n} + 1} = \frac{ - \kappa v(x_0) \cdot \hat{n} - \kappa u_{S = 0, \hat{n}}(x_0) \cdot \hat{n} - 1}{ \kappa u_{S = 0, \hat{n}}(x_0) \cdot \hat{n} + 1} - 1
\end{equation}
or
\begin{equation}
	f_0 = \frac{ - \kappa u_{S = 1, \hat{n}}(x_0) \cdot \hat{n} - 1}{\kappa u_{S = 0, \hat{n}}(x_0) \cdot \hat{n} + 1} - 1.
\end{equation}
We recognize $u_{S = 0, \hat{n}}(x_0) \cdot \hat{n}$ and $u_{S = 1, \hat{n}}(x_0) \cdot \hat{n} $ as the unactuated and actuated compliances under loading $f = \delta(x - x_0) \hat{n} $. Thus, the workpiece objective can then be written as a function of compliancies,
\begin{equation}
	\inf_{\Phi \in \mathcal{D}} \mathcal{O}(\Phi) = \frac{\kappa \mathcal{C}_{\hat{n}}(\Phi, 1) + 1}{\kappa \mathcal{C}_{\hat{n}}(\Phi, 0) + 1}
\end{equation}

\end{document}